\documentclass[journal,10pt]{IEEEtran} 
\usepackage{tikz}
\usetikzlibrary{positioning}
\usetikzlibrary{decorations.pathreplacing, shapes, arrows}
\usepackage{balance}
\usepackage{algpseudocode}
\usepackage[ruled,vlined,linesnumbered]{algorithm2e}
\SetKwInput{KwInput}{Input}
\SetKwInput{KwOutput}{Output}
\usepackage{cite}
\usepackage{amsfonts}
\usepackage{amsmath,amssymb,amsfonts,bm}

\usepackage{booktabs}
\usepackage{textcomp}
\usepackage{graphicx} 
\usepackage{epsfig}
\usepackage{amsthm}
\usepackage{subcaption}
\usepackage{mathrsfs}
\usepackage{times} % assumes new font selection scheme installed
\usepackage{amsmath} % assumes amsmath package installed
\DeclareMathOperator{\col}{col}
\DeclareMathOperator{\blkdiag}{blkdiag}
\usepackage{amssymb}  % assumes amsmath package installed
\usepackage{epstopdf}
\usepackage{array}
\usepackage{color}
\usepackage{mathtools}

\newtheorem {theorem}{Theorem}

\newtheorem {lemma}{Lemma}
\newtheorem {definition}{Definition}

\newtheorem {remark}{Remark}
 
\def\BibTeX{{\rm B\kern-.05em{\sc i\kern-.025em b}\kern-.08em
    T\kern-.1667em\lower.7ex\hbox{E}\kern-.125emX}}
\usepackage{url}
\usepackage{tikz} 

\usetikzlibrary{shapes,shapes.geometric,arrows,fit,calc,positioning,automata,}

\renewcommand{\SetKwInOut}[2]{%
	\sbox\algocf@inoutbox{\KwSty{#2}\algocf@typo:}%
	\expandafter\ifx\csname InOutSizeDefined\endcsname\relax% if first time used
	\newcommand\InOutSizeDefined{}%
	\sbox\algocf@inoutbox{\KwSty{#2}\algocf@typo\textbf{:}~}\setlength{\inoutindent}{\wd\algocf@inoutbox}%
	\else% else keep the larger dimension
	\ifdim\wd\algocf@inoutbox>\inoutsize%
	\sbox\algocf@inoutbox{\KwSty{#2}\algocf@typo\textbf{:}~}\setlength{\inoutindent}{\wd\algocf@inoutbox}%
	\fi%
	\fi% the dimension of the box is now defined.
	\algocf@newcommand{#1}[1]{%
		\ifthenelse{\boolean{algocf@inoutnumbered}}{\relax}{\everypar={\relax}}        {\let\\\algocf@newinout\hangindent=\inoutindent\hangafter=1\KwSty{#2}\algocf@typo\textbf{:}~##1\par}%
		\algocf@linesnumbered% reset the numbering of the lines
}}
\usepackage{titlesec}
\usepackage{tabularx}
\newcolumntype{L}{>{\raggedright\arraybackslash}p{0.32\columnwidth}}
\newcolumntype{R}{>{\raggedright\arraybackslash}X}

\begin{document}
\title{Consensus-Aware Co-design of Wireless Multi-Agent Control and Communication Systems}
\author{Atefeh Termehchi, \IEEEmembership{Member, IEEE}, Sajedeh Soleimani, Shiva Kazemi Taskooh, Mehdi Rasti, \IEEEmembership{Senior Member, IEEE}, and Ekram Hossain,  \IEEEmembership{Fellow, IEEE}
%\affil{Department of Electrical and Computer Engineering at the University of Manitoba, Winnipeg, Canada}
\thanks{Atefeh Termehchi, Shiva Kazemi Taskooh, and Ekram Hossain are with the Department of Electrical and Computer Engineering at the University of Manitoba, Winnipeg, Canada (emails: \{atefeh.termehchi, shiva.kazemitaskooh, ekram.hossain\}@umanitoba.ca). Sajedeh Soleimani is with Amirkabir University of Technology, Tehran, Iran (email: sajedehsoleimani@gmail.com). Mehdi Rasti is with the Centre for Wireless Communications, University of Oulu, Finland (email:{Mehdi.Rasti}@oulu.fi).}
}
\maketitle
\begin{abstract}
Autonomous multi-agent control systems (MACSs) enable distributed coordination in emerging applications such as smart power grids and autonomous vehicle or drone networks. In many of these applications, inter-agent communication relies on wireless networks. However, conventional MACS designs neglect the interaction between the communication and control subsystems. At the same time, conventional wireless designs reserve communication resources based on worst-case MACS requirements. Consequently, control performance is degraded and wireless resources are utilized inefficiently. To address these challenges, we propose a consensus-aware co-design approach for MACSs and wireless communication systems. We introduce a system model in which agents achieve consensus by exchanging information through device-to-device (D2D) links that share spectrum with cellular users under limited bandwidth resources. Based on this model, we formulate a joint optimization problem for consensus-topology selection and wireless resource allocation to minimize bandwidth consumption while guaranteeing consensus and satisfying the quality-of-service (QoS) requirements of both MACSs and cellular users. Although the resulting problem is a mixed-integer nonlinear program (NP-hard in general), a structural property of the optimal solution enables a sequential two-stage optimization algorithm with polynomial complexity. Simulations on a smart power grid case study demonstrate that the proposed approach achieves 24–27\% bandwidth savings compared with the two baselines while maintaining consensus and satisfying all QoS requirements of the control agents and the communication users.
\end{abstract}
\begin{IEEEkeywords}
Wireless multi-agent control, communication-control co-design,  
device-to-device (D2D) and cellular communication, consensus, quality-of-service 
%, bandwidth consumption.
\end{IEEEkeywords}
%\noindent\textbf{Remark (Delay-free model justification).}
%Although D2D communication introduces transmission delays, these delays are implicitly accounted for through constraint~C3, which enforces
%\[R_{ij}(T_k)\geq R_{\min}^{\mathrm{CA}}\]
%for every active D2D link. Accordingly, the packet transmission delay is bounded by
%\[\tau_{\max}=\frac{L_{\mathrm{pkt}}}{R_{\min}^{\mathrm{CA}}},\]
%where $L_{\mathrm{pkt}}$ denotes the control-packet size. Since $\tau_{\max}\ll T_s$, the communication delay is negligible relative to the load-frequency-control regulation time scale. Therefore, the delay-free model in \eqref{eq:stacked-dyn} provides an accurate approximation \cite{olfati2004consensus}.
\section{Introduction}
Autonomous multi-agent control systems (MACSs) underpin a wide range of emerging applications. In these applications, spatially distributed agents collaborate to perform real-time tasks, often without centralized coordination. Representative examples include coordinated power generation and load balancing among distributed energy resources (DERs), formation control and collision avoidance in autonomous vehicle platoons, and cooperative sensing and communication coverage by uncrewed aerial vehicles (UAVs). A core requirement across MACSs is {\em consensus}, the ability of agents to agree on a common state (e.g., load, position, and sensing) via repeated local information exchange with neighbors \cite{olfati2007consensus, ren2008distributed}. Unlike classical architectures that rely on wired or dedicated communication links with guaranteed quality-of-service (QoS), modern MACSs increasingly operate over shared wireless networks that coexist with other spectrum users. This transition reduces deployment and maintenance costs while improving flexibility and scalability \cite{park2017wireless}. However, it also introduces fundamental challenges at the control--communication interface, including the direct impact of wireless communication design on control performance and the influence of control design on wireless resource consumption. %Consequently, designing the wireless communication system and the MACS independently, without accounting for their strong two-way coupling, can significantly degrade both control performance and wireless resource efficiency \cite{han2023wireless}.

Traditionally, MACS designs abstract information exchange among agents as edges in a connectivity graph, while concrete wireless characteristics, such as bandwidth, channel quality, and resource allocation, are not explicitly modeled. Instead, multi-agent control designs typically characterize wireless communication using abstract metrics, such as delay and packet-loss bounds, while assuming that the underlying wireless network guarantees these bounds \cite{chen2019control}. Although such abstractions simplify control design, they do not reveal the wireless resources required to satisfy these communication requirements or how wireless resource allocation influences the resulting control performance. Meanwhile, wireless resources are typically provisioned conservatively based on worst-case communication requirements with additional reliability margins, resulting in inefficient bandwidth utilization. However, this approach is becoming increasingly unsustainable as bandwidth-intensive services, such as augmented reality/virtual reality and massive Internet of Things deployments, compete for limited wireless resources.
%Traditionally, MACS designs abstract the information exchange among agents as edges in a connectivity graph. However, concrete wireless-link specifications, such as bandwidth, channel quality, and resource allocation, are not explicitly modeled. Specifically, the multi-agent control designs typically account only for wireless impairments such as packet drops and delays. These designs typically rely on worst-case assumptions (e.g., delay bounds, loss rates) that the network infrastructure is presumed to guarantee \cite{chen2019control}. Nevertheless, the growing number of bandwidth-intensive applications (e.g., augmented reality/virtual reality streaming and massive Internet of Things deployments) makes conservative resource reservation for MACS increasingly unsustainable. 

Dynamic transmission scheduling approaches, such as joint (windowed) connectivity, self-triggered control, and event-triggered control, can reduce communication overhead in control systems. However, without careful co-design with the wireless network, the potential benefits of these approaches cannot be fully realized. In particular, insufficient wireless resources during periods of high communication demand may violate the communication requirements imposed by the control system, resulting in degraded control performance or even system instability \cite{tiberi2013energy}. Furthermore, these approaches alone cannot fully reduce wireless resource consumption because conventional wireless communication designs reserve bandwidth according to worst-case control requirements rather than the actual communication demand \cite{nowzari2019event}. For example, in MACSs, joint (windowed) connectivity relaxes the requirement for continuous communication among agents by allowing the communication topology to vary over time while preserving consensus (Fig.~\ref{joint}). Specifically, consensus is guaranteed as long as the union of the communication graphs over a finite time window remains connected \cite{ren2005consensus,moreau2005stability}. However, realizing these bandwidth savings requires joint design with the wireless network. Otherwise, consensus may be compromised due to insufficient wireless resources, and the potential bandwidth savings may be lost because resources are still provisioned for the worst case.
%This flexibility lowers resource consumption but complicates resource allocation, since the wireless network must dynamically ensure sufficient connectivity in every window. %Therefore, dynamic transmission scheduling approaches (e.g., self-triggered and event-triggered control) must jointly consider both the consensus and stability requirements of the control system and the resource constraints of the shared wireless network \cite{tiberi2013energy}.

Therefore, two critical research gaps must be addressed in the area of wireless-enabled MACSs: First, developing models for wireless-enabled MACS that explicitly capture two-way coupling between the multi-agent control and wireless communication; Second, co-designing multi-agent control and communication variables 
 to jointly satisfy dynamic requirements of control agents (CAs), the QoS needs of other wireless users, and the wireless network’s resource constraints. To address these gaps, we first develop an integrated modeling framework that captures the interaction between MACSs and wireless communication. Specifically, we consider a scenario in which CAs exchange information over existing wireless infrastructure (e.g., 6G) via device-to-device (D2D) links that share spectrum with cellular users (CUs) under limited bandwidth resources. Based on this framework, we formulate a joint optimization problem for consensus-topology selection and wireless resource allocation to minimize bandwidth consumption while guaranteeing consensus and satisfying the QoS requirements of both MACSs and CUs. 
 %In this system model, the QoS requirement of the MACSs is characterized by the admissible communication delay that preserves consensus. 
\begin{figure}[t]
    \centering
    \resizebox{0.85\columnwidth}{!}{%
    \begin{tikzpicture}[
        node/.style={circle, draw, inner sep=0.7pt, minimum size=7pt, font=\scriptsize},
        edget/.style={thick},
        edgu/.style={thick},
        title/.style={font=\scriptsize}
    ]

    %%% Compact positions %%%
    \def\xA{0}
    \def\xB{2.6}
    \def\xC{5.2}
    \def\xU{7.8}
    \def\y{0}

    %------------------ G(t1) ------------------%
    \node[title] at (\xA,1.2) {$\mathcal G(t_1)$};

    \node[node] (a1) at (\xA-0.6,\y+0.4) {$1$};
    \node[node] (a2) at (\xA+0.6,\y+0.4) {$2$};
    \node[node] (a3) at (\xA-0.3,\y-0.5) {$3$};
    \node[node] (a4) at (\xA+0.7,\y-0.5) {$4$};

    \draw[edget] (a1) -- (a2);
    \draw[edget] (a2) -- (a4);

    %------------------ G(t2) ------------------%
    \node[title] at (\xB,1.2) {$\mathcal G(t_2)$};

    \node[node] (b1) at (\xB-0.6,\y+0.4) {$1$};
    \node[node] (b2) at (\xB+0.6,\y+0.4) {$2$};
    \node[node] (b3) at (\xB-0.3,\y-0.5) {$3$};
    \node[node] (b4) at (\xB+0.7,\y-0.5) {$4$};

    \draw[edget] (b1) -- (b3);
    \draw[edget] (b3) -- (b4);

    %------------------ G(t3) ------------------%
    \node[title] at (\xC,1.2) {$\mathcal G(t_3)$};

    \node[node] (c1) at (\xC-0.6,\y+0.4) {$1$};
    \node[node] (c2) at (\xC+0.6,\y+0.4) {$2$};
    \node[node] (c3) at (\xC-0.3,\y-0.5) {$3$};
    \node[node] (c4) at (\xC+0.7,\y-0.5) {$4$};

    \draw[edget] (c2) -- (c3);

    %------------------ Union graph ------------------%
    \node[title] at (\xU,1.2) {$\bigcup_{k=1}^3 \mathcal G(t_k)$};

    \node[node] (u1) at (\xU-0.6,\y+0.4) {$1$};
    \node[node] (u2) at (\xU+0.6,\y+0.4) {$2$};
    \node[node] (u3) at (\xU-0.3,\y-0.5) {$3$};
    \node[node] (u4) at (\xU+0.7,\y-0.5) {$4$};

    % edges that appeared in some G(t_k)
    \draw[edgu] (u1) -- (u2);
    \draw[edgu] (u2) -- (u4);
    \draw[edgu] (u1) -- (u3);
    \draw[edgu] (u3) -- (u4);
    \draw[edgu] (u2) -- (u3);

    \end{tikzpicture}%
    }
    \caption{Joint connectivity: each graph $\mathcal G(t_k)$ may be
    disconnected, but their union is connected.}
    \label{joint}
\end{figure}
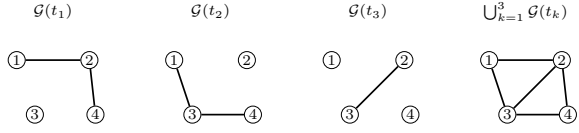
\subsection{Related Work} 
%This growing interest is driven by the tight coupling between wireless and control in CPS applications (e.g., smart grids, autonomous vehicle platoons, and electronic-health) where disjoint designs are impractical \cite{qiao2022communication}. 
In recent years, joint approaches that consider control and wireless communication systems, ranging from single-loop to large-scale multi-loop, have become an active research topic \cite{wang2025communication, eisen2019control, zeng2019joint, termehchi2020joint,  pang2025communication,  li2025distributed, diao2025task}. Generally, joint approaches can be categorized into three classes: control-aware wireless communication design, wireless-aware control design, and co-design of control and wireless communication. In control-aware wireless design, wireless variables (e.g., bandwidth and power allocation) are optimized under constraints that capture both control system requirements (e.g., stability) and communication requirements. For example, \cite{eisen2019control} maps control stability requirements to packet-delivery-rate constraints and jointly optimizes time–frequency scheduling and modulation-and-coding-scheme selection to guarantee low-latency and reliable wireless control. In wireless-aware control design, control variables (e.g., control inputs and control update rules) are determined to satisfy constraints that include requirements from control and wireless systems. For instance, \cite{zeng2019joint} studies control stability over a vehicle-to-vehicle wireless network with stochastic interference and Nakagami fading. The authors tune two control gains to satisfy delay constraints. However, in the co-design class, both wireless and control variables are jointly optimized while accounting for constraints from both domains, as in \cite{termehchi2020joint, pang2025communication,li2025distributed, diao2025task}. %, lima2022model,li2025novel}.
For example, the authors in \cite{termehchi2020joint} jointly optimize the control update rule and wireless resource allocation while satisfying control system stability and CUs QoS requirements. Similarly, the authors in \cite{pang2025communication} jointly optimize channel allocation and control input under control performance and communication resource constraints.

Although many studies have examined joint approaches to control and wireless communication, only a limited number specifically address the integration of multi-agent control with realistic wireless communication systems. In \cite{li2016event}, authors develop an event-triggered distributed average consensus algorithm for directed networks under strict communication bandwidth limits. Specifically, they use dynamic quantization to guarantee convergence with finite-bit transmissions. %The paper provides sufficient conditions to avoid quantizer saturation and shows that consensus can still be achieved even when each message contains only a small number of bits. 
The paper develops a mathematically rigorous dynamic quantizer for event-triggered consensus over digital channels; such quantization schemes can be incorporated into wireless-enabled MACSs. Authors in \cite{hong2019adaptive} propose a wireless-aware control design method in which each agent adaptively adjusts its communication range to maintain connectivity while significantly reducing communication power consumption during multi-agent consensus. However, the paper employs a simplified communication power-cost model, namely $w_i[k] = \epsilon\, d_i[k]^{\beta}$, which depends only on the communication range. In \cite{oliveira2021co}, the authors consider a dynamic multi-agent platoon composed of one leader and $n$ followers that communicate over vehicle ad hoc network (VANET) links. %Although the paper considers time-varying delays and packet losses at the communication level,  wireless characteristic (wireless interference) is incorporated only in the simulation setup.
Specifically, they propose a multi-agent controller that adapts to the wireless performance, making this paper a wireless-aware control design. To the best of our knowledge, there is no prior work on control-aware wireless design or co-design of control and wireless communication in the context of wireless-enabled MACSs. Consequently, existing studies do not (i) present a modeling framework that captures the tight, bidirectional coupling between MACSs and the wireless communication system, and (ii) co-design multi-agent control and wireless communication. %In addition, any such design must be inherently scalable with respect to the number of agents and wireless users. This requirement remains largely unaddressed and motivates the design developed in this paper.
\subsection{Contributions}
The main contributions of this paper are summarized as follows:
\begin{itemize}
\item We develop an integrated modeling framework for wireless-enabled MACSs that explicitly captures the two-way coupling between multi-agent control and wireless communication. In addition, we formulate a consensus-aware co-design optimization problem that jointly determines the time-varying consensus topology and wireless resource allocation to minimize bandwidth consumption while satisfying consensus and all QoS requirements.
\item  Despite the NP-hardness of the resulting mixed-integer nonlinear optimization problem, a structural property
of the optimal solution, formalized in \textbf{Lemma 2}, enables the feasible search space to be reduced without loss of optimality. This structural property provides theoretical insight into the interaction between the communication and control subsystems by showing that their joint optimization can be solved in two sequential stages with polynomial complexity per slot.
%We show that, despite the NP-hardness of the resulting mixed-integer nonlinear optimization problem, a structural property of the optimal solution enables a sequential two-stage optimization algorithm with polynomial complexity.
\item Simulation results on a smart power grid case study that coexists with CUs demonstrate significant bandwidth savings compared with existing approaches while maintaining consensus and satisfying all QoS requirements.
\end{itemize}

\section{System Model and Assumptions}\label{sec:system-model}

We consider a system model comprised of a base station (BS), $N_c$ cooperative CAs that together form a MACS, and $N_m$ mobile CUs. Let $\mathcal{N}_c \triangleq \{1,\dots,N_c\}$ denote the set of CAs and $\mathcal{N}_m \triangleq \{1,\dots,N_m\}$ denote the set of CUs. It is assumed that the $N_c$ CAs (e.g., robots, distributed energy resources) exchange state and control information over wireless D2D links. The interaction topology among the CAs is modeled as a time-varying undirected graph. We assume that the total system bandwidth $\omega$ is shared between the CAs and the CUs. All nodes (CAs and CUs) can transmit simultaneously, but they are allocated disjoint portions of the bandwidth. That is, each node is assigned a non-overlapping sub-band of $\omega$, so that no two simultaneous transmissions occupy overlapping frequency bands, thereby preventing interference in the frequency domain. In this system, the BS is responsible for allocating bandwidth subject to (i) consensus requirement of the MACS, (ii) minimum QoS requirements of the CAs and CUs, and (iii) total bandwidth budget limitation. Furthermore, time is partitioned into slots $T_k$, $k \in \mathbb{N}$. We denote by $T$ the total slots in the considered horizon. Within each slot $T_k$, the resource allocation and the MACS communication graph are assumed to be constant; however, across slots, they may vary. 

\subsection{Multi-Agent Control System Dynamics and Network Coupling}
Over time $t\in [T_k,T_{k+1})$, \reflectbox the dynamics of agent $i \in \mathcal{N}_c$ is modeled as:
\begin{align}
\dot{\mathbf x}_i(t)
&= \mathbf A_i^d\, \mathbf x_i(t)+ \mathbf{B}_i\,\mathbf{u}_i(t) \;+\;\mathbf E_i(t) \nonumber\\&\;+\;\sum_{j \in \mathcal{N}_c \setminus \{i\}} a_{ij}(T_k)\;\mathbf C\big(\mathbf x_j(t)- \mathbf x_i(t)\big),
\label{eq:local-dyn}
\end{align}
where variable $\mathbf{x}_i(t)\in\mathbb{R}^q$ denotes the state of agent $i$
at time $t$ and $q$ is number of state variables. A dot above a variable denotes its derivative with respect to time.  $\mathbf A_i^d\in\mathbb{R}^{q\times q}$  is state matrix of each agent $i$, $\mathbf C\in\mathbb{R}^{q\times q}$ is the coupling/consensus gain, $\mathbf B_i\in\mathbb{R}^{q\times m_i}$ maps control input to state, $m_i$ is number of input variables, and $\mathbf E_i(t)\in\mathbb{R}^{q}$ is an exogenous disturbance (e.g., modeling process noise). 
Besides, inter-agent information exchange uses D2D wireless links and is described by a time-varying undirected graph:
\begin{equation}
\mathcal{G}(T_k) \;=\; \big(\mathcal{V},\,\mathcal{E}(T_k)\big), 
\quad \mathcal{V}=\mathcal{N}_c,
\end{equation}
where $\mathcal{E}(T_k)\subseteq \mathcal{V}\times\mathcal{V}$ is the edge set, with symmetric
adjacency matrix $\mathbf A(T_k)=[a_{ij}(T_k)]\in\{0,1\}^{N_c\times N_c}$. For the undirected graph, $a_{ii}(T_k)=0$ and
$a_{ij}(T_k)=a_{ji}(T_k)$. Moreover, $a_{ij}(T_k)=1$ if and only if,
within slot $T_k$, $t\in [T_k,T_{k+1})$, the allocated D2D resources between agents $i$ and $j$
are sufficient to support mutual information exchange while satisfying the
admissible transmission-delay bound%or above the rate requirement $R_{\min}^{\mathrm{CA}}$and the corresponding delay is admissible, $\tau_{max}$,
; otherwise, $a_{ij}(T_k)=0$. This condition is formalized in \textbf{Lemma 1}.  
%Note that since we assume inter-agent communication graph $\mathcal{G}(T_k)$ is an undirected graph so $a_{ij}(T_k)=a_{ji}(T_k)$. 
The model in \eqref{eq:local-dyn} is a continuous-time multi-agent dynamics with piecewise-constant switching communication topology. The overall MACS in \eqref{eq:local-dyn} can be interpreted as a hybrid/switched system.

For model in \eqref{eq:local-dyn}, the graph Laplacian is defined as:
\begin{equation}
\mathbf L(T_k)\;=\; \mathbf \Delta(T_k)- \mathbf A(T_k),
\end{equation}
where $\mathbf \Delta(T_k)$ is the degree matrix and is defined as:
\[
\mathbf \Delta(T_k)=\mathrm{diag} \big(d_1(T_k),\dots,d_{N_c}(T_k)\big),
\quad d_i(T_k)=\sum_{j=1}^{N_c} a_{ij}(T_k).
\] Indeed, $a_{ij}(T_k)$ indicates the presence of a link and $d_i(T_k)$ counts how many neighbors agent $i$ has in slot $T_k$. In addition, matrix $\mathbf L(T_k)$ encodes the connectivity of the MACS. 
Let 
$\mathbf X(t)\!\coloneqq\!\col\!\big(\mathbf x_1(t),\dots,\mathbf x_{N_c}(t)\big)\in\mathbb{R}^{N_c q}$ and
$\mathbf E(t)\!\coloneqq\!\col\!\big(\mathbf E_1(t),\dots,\mathbf E_{N_c}(t)\big)$,
where $\col(\cdot)$ stacks its arguments column-wise. 
Additionally, let the block-diagonal matrices 
$\mathbf A^d\!\coloneqq\!\blkdiag(\mathbf A^d_1,\dots,\mathbf A^d_{N_c})\in\mathbb{R}^{N_c q\times N_c q}$, and $\mathbf B\!\coloneqq\!\blkdiag(\mathbf B_1,\dots,\mathbf B_{N_c})$. 
Using the Kronecker product $\otimes$, the stacked dynamics of the wireless-enabled MACS is \cite{olfati2007consensus,ren2008distributed}:%,zhang2011optimal
\begin{equation}
\dot{\mathbf X}(t)
=\Big(\mathbf A^d - \mathbf L(T_k)\!\otimes\! \mathbf C\Big)\,\mathbf X(t) + \mathbf B \mathbf U(t) +\mathbf E(t)\,
\qquad \forall t\in T_k,
\label{eq:stacked-dyn}
\end{equation}
where $\mathbf L(T_k)\in\mathbb{R}^{N_c\times N_c}$ is the graph Laplacian in slot $T_k$, $\mathbf C\in\mathbb{R}^{q\times q}$ is the coupling gain, and $\mathbf L(T_k)\!\otimes\! \mathbf C\in\mathbb{R}^{N_c q\times N_c q}$, matching the dimension of $\mathbf A^d$.
Here, the term $-\,\mathbf L(T_k)\!\otimes \mathbf C$ enforces consensus-type coupling: if two agents are connected ($a_{ij}(T_k)=1$), their states are driven toward agreement along the directions weighted by $\mathbf C$; if they are not connected, no direct coupling is applied in that slot.

\vspace{0.2cm}
\noindent\textbf{Assumption 1}. Each agent is internally stable; equivalently, the state matrix $\mathbf A_i^d$ is Hurwitz, meaning that all eigenvalues of $\mathbf A_i^d$ lie in the open left half-plane.
\subsection{Consensus in Multi-Agent Control System}
The objective of the consensus in a MACS is to make all control agents agree on a
common state despite having different initial conditions and local disturbances.
To quantify the consensus performance, the overall state of the MACS is decomposed
into two orthogonal components: the average component, representing the common
motion of all agents, and the disagreement component, representing the deviation
of individual agents from the average.
Let $\mathbf{1}\in\mathbb{R}^{N_c}$ denote the all-ones vector and define the
disagreement projector as
\[\mathbf \Pi
\triangleq
\mathbf I_{N_c}
-
\frac{1}{N_c}\mathbf 1\mathbf 1^\top .\]
For the stacked state
$\mathbf X(t)=\col(\mathbf x_1(t),\ldots,\mathbf x_{N_c}(t))
\in\mathbb{R}^{N_cq}$,
the average and disagreement vectors are defined, respectively, as
\begin{align}
\bar{\mathbf x}(t)
&=
\frac{1}{N_c}
\sum_{i=1}^{N_c}\mathbf x_i(t)
=
\left(
\frac{1}{N_c}\mathbf 1^\top\otimes\mathbf I_d
\right)\mathbf X(t),
\\
\tilde{\mathbf X}(t)
&=
(\mathbf\Pi\otimes\mathbf I_d)\mathbf X(t).
\label{disagree}
\end{align}
Consensus is achieved if the disagreement vanishes asymptotically,
\[
\|\tilde{\mathbf X}(t)\|
\rightarrow
0,
\qquad
t\rightarrow\infty,
\]
which is equivalent to
\(
\mathbf x_1(t)=\mathbf x_2(t)=\cdots=\mathbf x_{N_c}(t)
\rightarrow
\bar{\mathbf x}(t).\)
Using \eqref{eq:stacked-dyn} and the Laplacian property
$\mathbf 1^\top\mathbf L(T_k)=0$, the consensus coupling does not influence the
average part. Consequently, the average state evolves according to
\begin{equation}
\dot{\bar{\mathbf x}}(t)
=
\bar{\mathbf A}^d\bar{\mathbf x}(t)
+
\bar{\mathbf E}(t),
\label{eq:avg-dynamics}
\end{equation}
where
$\bar{\mathbf A}^d
=
\frac{1}{N_c}
\sum_{i=1}^{N_c}\mathbf A_i^d$
and
$\bar{\mathbf E}(t)
=
\frac{1}{N_c}
\sum_{i=1}^{N_c}\mathbf E_i(t)$.
Therefore, the role of the consensus coupling
$\mathbf L(T_k)\otimes\mathbf C$
is solely to reduce the disagreement
$\tilde{\mathbf X}(t)$, thereby driving all agents toward the common network
trajectory.

In what follows, we first define a connected undirected graph and its algebraic connectivity, then introduce the union graph over a window of slots. These concepts are used to characterize the conditions on the communication graph under which consensus is guaranteed in the MACS.
\begin{definition}[Connected undirected graph]
Consider an undirected graph $\mathcal{G} = (\mathcal{V}, \mathcal{E})$, where 
$\mathcal{V} = \{1,\dots,N_c\}$ is the set of agents and 
$\mathcal{E} \subseteq \mathcal{V} \times \mathcal{V}$ is the set of undirected edges.
We say that $\mathcal{G}$ is connected if, for every pair of distinct agents
$i,j \in \mathcal{V}$, there exists a path from $i$ to $j$, i.e., a finite sequence of agents
\(
i = v_0, v_1, v_2, \dots, v_q = j
\)
such that $(v_{p-1}, v_p) \in \mathcal{E}$ for all $p = 1,\dots,q$.
Equivalently, all agents belong to a single connected component.
\end{definition}

\begin{definition}[Algebraic connectivity \cite{olfati2007consensus}]
For a symmetric Laplacian matrix $\mathbf L\in\mathbb{R}^{N_c\times N_c}$ with eigenvalues
$0=\lambda_1(\mathbf{L})\le\lambda_2(\mathbf{L})\le\cdots\le\lambda_{N_c}(\mathbf{L})$,
the scalar $\lambda_2(\mathbf{L})$ is called the algebraic connectivity of the graph.
For undirected graphs with nonnegative edge weights, $0<\lambda_2(\mathbf{L})$ if and only if the graph is connected.
\end{definition}

\begin{definition}[Union graph over a window of slots {\cite{jadbabaie2003coordination}}]
Let $\mathcal{G}(T_k)=(\mathcal{V},\mathcal{E}(T_k))$ denote the undirected inter-agent communication graph in slot $T_k$ with Laplacian $\mathbf{L}(T_k)$.
For an integer $K\ge 1$, union graph over slots $T_k,\dots,T_{k+K-1}$ is:
\[
\mathcal{G}_{[k:k+K-1]}
\;\triangleq\;
\big(\mathcal{V},\,
\bigcup_{\ell=k}^{k+K-1} \mathcal{E}(T_\ell)\big),
\]
and its aggregated Laplacian is
\[
\mathbf{L}_{[k:k+K-1]}
\;\triangleq\;
\sum_{\ell=k}^{k+K-1} \mathbf{L}(T_\ell).
\]
For the undirected graph, $0<\lambda_2(\mathbf{L}_{[k:k+K-1]})$ if and only if the union graph $\mathcal{G}_{[k:k+K-1]}$ is connected.
\end{definition}

We now characterize when consensus is guaranteed.

\begin{theorem}
\label{thm:consensus-windowed}
 {\cite{ren2008distributed,scardovi2008synchronization}}
Consider the wireless-enabled MACS with stacked dynamics (\ref{eq:stacked-dyn}) and suppose the following conditions hold:
\begin{itemize}
    \item[(A1)] \textbf{Internal stability of agents.}
    Each agent is internally stable: $\mathbf A^d$ is Hurwitz.
    %Equivalently, there exists $P_i \succ 0$ such that $(A_i^d)^\top P_i + P_i A_i^d \prec 0$.
    \item[(A2)] \textbf{Positive coupling gain.}
    The coupling gain $\mathbf C$ is symmetric and positive definite, with:
    \(
    \mathbf C \succeq \underline{c}\, \mathbf I_d
    \quad \text{for some }\underline{c} > 0.
    \)
    \item[(A3)] \textbf{Joint (windowed) connectivity.}
     There exist an integer $\tilde{K} \ge 1$ and a constant $\underline{\lambda}_{\tilde{K}}>0$ such that,
     for any starting index $k \in \mathbb{N}$,
    \begin{equation}
    \lambda_2\!\Bigg(\sum_{\ell=k}^{k+\tilde{K}-1} \mathbf L(T_\ell)\Bigg)
    \;\ge\;
    \underline{\lambda}_{\tilde{K}}.
    \label{eq:windowed-lambda2-thm}
    \end{equation}
    Equivalently, the union graph over any $\tilde{K}$ consecutive slots is connected.
\end{itemize}
Then, in the disturbance-free case $\mathbf E(t)\equiv 0$, the disagreement $\tilde {\mathbf{X}}(t) = (\mathbf \Pi\!\otimes\!\mathbf{I}_d)\mathbf{X}(t)$
converges to zero exponentially. In particular,
\[
\|\tilde{\mathbf{X}}(t)\|
\;\le\;
c\,e^{-\alpha t}\,\|\tilde{\mathbf{X}}(0)\|
\quad \text{for some } c>0,\ \alpha>0,
\]
and thus:
\(
\mathbf x_1(t)=\mathbf x_2(t)=\cdots= \mathbf x_{N_c}(t)
\;\longrightarrow\;
\bar{\mathbf{x}}(t),
\)
where $\bar{\mathbf{x}}(t)$ evolves according to averaged stable self-dynamics \eqref{eq:avg-dynamics}.
\end{theorem}
\begin{proof}
    See \cite{ren2008distributed,scardovi2008synchronization}.
\end{proof}
\noindent \textbf{Assumption 2}. The wireless-enabled MACS in \eqref{eq:stacked-dyn} satisfies conditions {(A1)} and {(A2)} of \textbf{Theorem~\ref{thm:consensus-windowed}}.

%\begin{remark}[Lower bound on algebraic connectivity]
%For undirected graphs with nonnegative edge weights, equation \eqref{eq:windowed-lambda2-thm} enforces that the aggregated union graph over each $\tilde{K}$-slot window is connected, and moreover provides a quantitative lower bound $\underline{\lambda}_{\tilde{K}}$ on its algebraic connectivity.
%\end{remark}

\begin{remark}[Effect of disturbances {\cite{ren2008distributed}}] Assume  $\mathbf E(t)$ is bounded, i.e., $\|\mathbf E(t)\|\le \bar E < \infty$ for all $t$.
Under (A1)--(A3), the disagreement dynamics are input-to-state stable:
$\tilde{\mathbf{X}}(t)$ remains bounded and can be made arbitrarily small by increasing either the coupling level $\underline{c}$ in (A2) or the windowed connectivity level $\underline{\lambda}_{\tilde{K}}$ in (A3).
\end{remark}
\begin{remark}[Consensus speed]
Condition (A3) guarantees that all agents eventually agree. We can also discuss how fast they agree. Under (A1)--(A3), the disagreement $\tilde{\mathbf{X}}(t)$ between
agents decays approximately like $e^{-\alpha t}$ for some $\alpha>0$, i.e., it shrinks
at an exponential rate. This rate $\alpha$ becomes larger (faster convergence) when:
\begin{itemize}
    \item the coupling gain $\mathbf C$ is stronger, i.e., $\underline{c}$ in (A2) is larger. Intuitively,
    each agent reacts more aggressively to differences between its state and its neighbors'
    states,
    \item the network is better connected over time, i.e., $\underline{\lambda}_{\tilde{K}}$ in (A3)
    is larger or the horizon $\tilde{K}$ is smaller. Intuitively, information can propagate across
    the entire set of agents more quickly because there are enough reliable D2D links
    activated across the short window of $\tilde{K}$ slots.
\end{itemize}
%In summary: stronger pairwise coupling and more persistent connectivity both make the agents reach consensus faster.
\end{remark}
\begin{remark}[On switching graphs]
Conditions (A1)--(A3) do not require each $\mathbf L(T_k)$ to correspond to a connected graph. Instead, they require that the union of the active D2D links over any window of $\tilde{K}$ consecutive slots be connected.
%to correspond to a connected graph. Instead, they require that, over any window of $\tilde{K}$ consecutive slots, the union of active D2D links is connected. %This is the standard uniformly jointly connected condition for consensus over switching topologies in MACSs; see, e.g., \cite{ren2008distributed,jadbabaie2003coordination}.
%\cite{RenBeardTAC2005,ShiJohanssonSIAMJCO2013}.
\end{remark}
\subsection{Wireless Communication Subsystem}
As mentioned earlier, the wireless communication system consists of one BS, $N_c$ CAs, and $N_m$ mobile CUs. We assume that CAs communicate via D2D links, and that cellular communication and D2D communication occur simultaneously. In addition, we assume that %the total system bandwidth $W>0$ (Hz), which is partitioned among all active links in the network during slot $T_k$. 
the total system bandwidth $\omega>0$ (Hz) is shared in each slot $T_k$ by (i) D2D links between CAs, and
(ii) BS--CU links (uplink or downlink). For each unordered CA pair $(i,j)$ with $i,j\in\mathcal{N}_c$ and $i<j$, we allocate a normalized bandwidth share
\(
b_{ij}(T_k)\in[0,1],
\)
which corresponds to an absolute bandwidth
\(
B_{ij}(T_k)\triangleq b_{ij}(T_k)\, \omega\quad\text{[Hz]}.
\)
For each cellular user $u\in\mathcal{N}_m$, we allocate a normalized share
\(
b_{u}(T_k)\in[0,1],
\)
with absolute bandwidth
\(
B_{u}(T_k)\triangleq b_{u}(T_k)\, \omega \quad\text{[Hz]}.
\)
All allocated sub-bands are assumed orthogonal (non-overlapping in frequency),
so simultaneously served links do not interfere in-band.
The total bandwidth used in slot $T_k$ must satisfy the per-slot bandwidth budget:
\begin{equation}
\sum_{\substack{i,j\in\mathcal{N}_c\\ i<j}} b_{ij}(T_k)
\;+\;
\sum_{u\in\mathcal{N}_m} b_{u}(T_k)
\;\le\; 1,\qquad \forall\,k.
\label{eq:BW-budget-ULDL}
\end{equation}
Furthermore, the channels between the CAs, as well as the channels between the BS and the CUs, are modeled as Rayleigh fading channels.
%since we consider no line of sight. 
%Although wireless channels typically exhibit temporal correlation, we assume that within each time slot $T_k$, the channel state remains constant. %Let $h_{ij}(T_k)$ be the complex channel gain between CAs $i$ and $j$ in slot $T_k$ and $h_{u}(T_k)$ denote the channel gain between CU $u$ and the BS in slot $T_k$. The $h_{ij}(T_k)$ and $h_{u}(T_k)$ are defined as: 
Let $h_{ij}(T_k)$ denote the complex channel gain between CAs $i$ and $j$
in slot $T_k$, for $i,j \in \mathcal{N}_c$, which
%$h_{ij}(T_k)$ 
can be defined as:
\begin{equation}
h_{ij}(T_k)
\;=\;
{ d_{ij}^{-\eta/2} }\;\tilde{h}_{ij}(T_k),
\label{eq:channel-CA-CA}
\end{equation}
where $d_{ij}$ is the physical distance between agents $i$ and $j$, $\eta > 0$ is the path-loss exponent, and $\tilde{h}_{ij}(T_k) \sim \mathcal{CN}(0,1)$ models small-scale fading. %Thus $d_{ij}^{-\eta}$ captures large-scale attenuation (path loss), and $\tilde{h}_{ij}(T_k)$ captures slot-$T_k$ small-scale fading.
In addition, let $h_u(T_k)$ denote the complex channel gain between the BS and CU $u \in \mathcal{N}_m$
in slot $T_k$.
Similarly, $h_u(T_k)$ can be defined as:
\begin{equation}
h_u(T_k)
\;=\;
d_{u}^{-\eta/2}(T_k)\;\tilde{h}_u(T_k),
\label{eq:channel-BS-CU}
\end{equation}
where $d_u(T_k)$ is the distance between CU $u$ and the BS in slot $T_k$, $\eta > 0$ is again the path-loss exponent, $\tilde{h}_u(T_k) \sim \mathcal{CN}(0,1)$ is the  small-scale fading coefficient in slot $T_k$.
%Both $\tilde{h}_{ij}(T_k)$ and $\tilde{h}_u(T_k)$ are assumed independent across different slots.

Suppose agent $i$ transmits to agent $j$ using constant transmit power $P_{CA}$ [W]. Then
the received signal-to-noise ratio (SNR) at agent $j$ is:
\begin{equation}
\mathrm{SNR}_{i \to j}(T_k)
\;=\;
\frac{
P_{CA}\,\big|h_{ij}(T_k)\big|^2
}{
N_0\,B_{ij}(T_k)
},
\label{eq:SNR-i-to-j}
\end{equation}
where $N_0$ [W/Hz] denotes the noise power spectral density and $h_{ij}(T_k)$ is given by \eqref{eq:channel-CA-CA}. Additionally, we assume that, within each slot $T_k$, the D2D exchange between agents $i$ and $j$ is half-duplex and time-shared. Hence, the slot is divided into two equal sub-intervals of duration $T_s/2$, one used for transmission $i \to j$ and the other for $j \to i$, while both sub-intervals reuse the same allocated bandwidth $B_{ij}(T_k)$.
Accordingly, the corresponding Shannon rate from $i$ to $j$ in slot $T_k$ is
\begin{align}
R_{i \to j}(T_k)
\;=\;
\frac{1}{2} B_{ij}(T_k)\;
\log_2 \!\Big( 1 + \mathrm{SNR}_{i \to j}(T_k) \Big). 
\label{eq:rate-i-to-j}    
\end{align}
Similarly, agent $j$ transmits to agent $i$ in the same slot with constant power $P_{CA}$ [W], we obtain
\begin{align}
R_{j \to i}(T_k)
&=
\frac{1}{2} B_{ij}(T_k)\;
\log_2\!\Big(1+\mathrm{SNR}_{j \to i}(T_k)\Big),
\label{eq:rate-j-to-i}
\end{align}
where $\mathrm{SNR}_{j \to i}(T_k)
=\frac{P_{CA}\,\big|h_{ij}(T_k)\big|^2
}{N_0\,B_{ij}(T_k)}$. Here, we have used the same complex gain $h_{ij}(T_k)$ for both directions, which is consistent with the symmetric (undirected graph) D2D model between agents $i$ and $j$. 

To prevent communication delays on D2D links from destabilizing
the consensus dynamics~\eqref{eq:stacked-dyn}, the one-way transmission delay on every
active link must not exceed the admissible delay margin. Following \cite{olfati2004consensus}, the admissible delay margin is:
\begin{equation}
\tau^{*}(T_k)
\;=\;
\frac{\pi}{2\,\lambda_{\max}\!\left(\mathbf{L}(T_k)\right)},
\label{eq:tau-star}
\end{equation}
where $\lambda_{\max}\!\left(\mathbf{L}(T_k)\right)$ denotes the largest
eigenvalue of the graph Laplacian $\mathbf{L}(T_k)$.
Specifically, if the transmission delay on every active D2D link
satisfies $\tau_{ij}(T_k) < \tau^{*}(T_k)$
\cite{olfati2004consensus}, the consensus dynamics
in~\eqref{eq:stacked-dyn} remain asymptotically stable despite
the transmission delay, and the delay-free model in~\eqref{eq:local-dyn} provides an accurate representation of the true system.
Based on this condition, the minimum CA rate $R_{\min}^{\mathrm{CA}}$ for all active links and all slots $T_k$
is established in the following lemma.

\begin{lemma}[Minimum CA Rate]
\label{lem:Rmin-CA}
Let $N_c = |\mathcal{N}_c|$ be the number of agents
in the D2D communication graph $\mathcal{G}(T_k)$,
and let $L_{\mathrm{pkt}}$ [bits] denote the control-packet
size. If the minimum CA--CA rate $R_{\min}^{\mathrm{CA}}$ satisfies
\begin{equation}
R_{\min}^{\mathrm{CA}}
\;>\;
\frac{4\,L_{\mathrm{pkt}}\,(N_c-1)}{\pi},
\label{eq:Rmin-CA-bound}
\end{equation}
then the packet transmission delay
$\tau_{ij}(T_k) = L_{\mathrm{pkt}}/R_{i\to j}(T_k)$
on every active D2D link satisfies
$\tau_{ij}(T_k) < \tau^{*}(T_k)$ for all slots $T_k$,
and the transmission delay does not destabilize the consensus dynamics in~\eqref{eq:stacked-dyn}.
\end{lemma}

\begin{proof}
By definition, link $(i,j)$ is active in slot $T_k$ only
if $R_{i\to j}(T_k) \geq R_{\min}^{\mathrm{CA}}$, bounding
the one-way transmission delay by
\begin{equation}
\tau_{ij}(T_k)
\;=\;
\frac{L_{\mathrm{pkt}}}{R_{i\to j}(T_k)}
\;\leq\;
\tau_{\max}
\;\coloneqq\;
\frac{L_{\mathrm{pkt}}}{R_{\min}^{\mathrm{CA}}}.
\label{eq:proof-tau-max}
\end{equation}
Applying the Gershgorin circle
theorem~\cite{olfati2004consensus} to each row $i$ of
$\mathbf{L}(T_k)$, with diagonal entries
$l_{ii}=d_i(T_k)\geq 0$ and off-diagonal entries
$l_{ij}=-a_{ij}(T_k)\leq 0$:
\begin{equation}
\big|\lambda - d_i(T_k)\big|
\;\leq\;
\sum_{j \neq i}|l_{ij}|
\;=\;
d_i(T_k),
\end{equation}
so all eigenvalues of $\mathbf{L}(T_k)$ lie in
$[0,\,2d_i(T_k)]$.
Since each agent has at most $N_c-1$ neighbors,
$d_{\max}(T_k) \leq N_c - 1$, giving
\begin{equation}
\lambda_{\max}\!\left(\mathbf{L}(T_k)\right)
\;\leq\;
2\,d_{\max}(T_k)
\;\leq\;
2\,(N_c-1).
\label{eq:proof-lmax-bound}
\end{equation}
Substituting into~\eqref{eq:tau-star}:
\(
\tau^{*}(T_k)
\;\geq\;
\frac{\pi}{4\,(N_c-1)}.
\label{eq:tau-star-bound}
\)
The condition $\tau_{\max} < \tau^{*}(T_k)$ then holds
for all slots $T_k$ if
\begin{equation}
\frac{L_{\mathrm{pkt}}}{R_{\min}^{\mathrm{CA}}}
\;<\;
\frac{\pi}{4\,(N_c-1)},
\end{equation}
which rearranges to~\eqref{eq:Rmin-CA-bound}.
Hence $\tau_{ij}(T_k) \leq \tau_{\max} < \tau^{*}(T_k)$
for all active links and all slots $T_k$.%, and the conclusion follows from~\cite{olfati2004consensus}.
\end{proof}
%Indeed, we require that both directions can deliver at least the state/control information payload $R_{\min}^{\mathrm{CA}} T_s$ per slot, where $T_s$ is slot duration. As $R_{i\to j}(T_k)$ and $R_{j\to i}(T_k)$ are rates (bits/s), we have:
%\begin{equation}
%R_{i \to j}(T_k) \;\ge\; R_{\min}^{\mathrm{CA}}\quad\text{and}\quadR_{j \to i}(T_k) \;\ge\; R_{\min}^{\mathrm{CA}}.
%\label{rate_con}
%\end{equation}
Therefore, for each unordered pair of distinct agents $(i,j)$ with $i\neq j$, we define the symmetric adjacency matrix's entry (i.e. $a_{ij}(T_k)=a_{ji}(T_k)$) according to whether the D2D resources allocated between agents $i$ and $j$ are sufficient to support mutual information exchange while satisfying the admissible transmission-delay bound during slot $T_k$:
%reliably exchange state/control information during slot $T_k$:
\begin{equation}
a_{ij}(T_k)=
\begin{cases}
1, & \displaystyle \min\!\big\{ R_{i \to j}(T_k),\, R_{j \to i}(T_k) \big\} \;\ge\; R_{\min}^{\mathrm{CA}},\\[4pt]
0, & \text{otherwise},
\end{cases}
\label{graph_element}
\end{equation}
for all $i\neq j$, and we set $a_{ii}(T_k)=0$. %Here $R_{i \to j}(T_k)$ and $R_{j \to i}(T_k)$ denote the achievable D2D rates (bits/s) from $i$ to $j$ and from $j$ to $i$, respectively, under the allocated bandwidths and fixed transmit powers in slot $T_k$.
For each active D2D pair $(i,j)$, the link is symmetric%and both directions share the same allocated bandwidth $B_{ij}(T_k)$ and equal transmit power $P^{\mathrm{CA}}$
, consequently:
\(
R_{i \to j}(T_k)=R_{j \to i}(T_k)\triangleq R_{ij}(T_k),\)
%\label{eq:symmetric-rate}
%\end{equation}
and the adjacency rule in \eqref{graph_element} simplifies:
\begin{equation}
a_{ij}(T_k)=
\begin{cases}
1, & R_{ij}(T_k)\ge R_{\min}^{\mathrm{CA}},\\[3pt]
0, & \text{otherwise},
\end{cases}
\quad a_{ii}(T_k)=0.
\label{eq:symmetric-adj}
\end{equation}
%\vspace{-25pt}
Now consider CU $u \in \mathcal{N}_m$ in the uplink direction (CU~$\to$~BS).
%The bandwidth allocated to CU $u$ for uplink in slot $T_k$ is
%\[
%B_u^{\mathrm{UL}}(T_k) \;=\; b_u^{\mathrm{UL}}(T_k)\,W \quad \text{[Hz]}.
%\]
Let $P_u^{\mathrm{UL}}$ [W] be the CU's uplink transmit power in that slot, and let $h_u(T_k)$ be the BS--CU channel in \eqref{eq:channel-BS-CU}.
Then the uplink receive SNR at the BS is:
%\begin{equation}
\(\mathrm{SNR}_u^{\mathrm{UL}}(T_k)
\;=\;
\frac{
P_u^{\mathrm{UL}}\,\big|h_u(T_k)\big|^2
}{
N_0\,B_u(T_k)
}.\)
%\label{eq:SNR-UL}
%\end{equation}
We assume half-duplex BS–CU links with equal time-sharing between uplink and downlink, then the corresponding achievable uplink rate (bits/s) is:
\begin{equation}
R_u^{\mathrm{UL}}(T_k)
\;=\;
\frac{1}{2}B_u(T_k)\;
\log_2\!\Big(
1 + \mathrm{SNR}_u^{\mathrm{UL}}(T_k)
\Big).
\label{eq:rate-UL}
\end{equation}
For the downlink (BS~$\to$~CU $u$), %the allocated downlink bandwidth in slot $T_k$ is
%\[
%%\]
let $P_u^{\mathrm{DL}}$ [W] denote the BS transmit power dedicated to CU $u$ over that band in slot $T_k$.
Using the same magnitude channel gain $|h_u(T_k)|^2$ from \eqref{eq:channel-BS-CU}, the downlink SNR at CU $u$ is:
%\begin{equation}
\(\mathrm{SNR}_u^{\mathrm{DL}}(T_k)
\;=\;
\frac{
P_u^{\mathrm{DL}}\,\big|h_u(T_k)\big|^2
}{
N_0\,B_u(T_k)
},\)
%\label{eq:SNR-DL}
%\end{equation}
and the corresponding downlink rate (bits/s) is:
\begin{equation}
R_u^{\mathrm{DL}}(T_k)
\;=\;
\frac{1}{2}B_u(T_k)\;
\log_2\!\Big(
1 + \mathrm{SNR}_u^{\mathrm{DL}}(T_k)
\Big).
\label{eq:rate-DL}
\end{equation}
%We assume half-duplex operation at the BS–CU links and uniform time-sharing between uplink and downlink. The factor \(1/2\) in \eqref{eq:rate-UL} and \eqref{eq:rate-DL} accounts for the equal time split between uplink and downlink.
%Note that the factor \(1/2\) in \eqref{eq:rate-UL} and \eqref{eq:rate-DL} arises from the assumption of equal time sharing.% between the uplink and downlink.
\vspace{-26pt}

\section{Problem Formulation}\label{sec:problem}
In Section~\ref{sec:system-model}, we introduced the system
model, consisting of a MACS and a wireless communication
system supporting CA--CA and BS--CU links.
We now formulate the joint consensus-topology and bandwidth-allocation optimization problem for the considered system. The objective is to minimize the total bandwidth consumption $B^{\mathrm{total}}(T_k)$ per slot.
To this end, we optimize the D2D topology
$\{a_{ij}(T_k)\}_{i<j}$ and the bandwidth allocation
$\mathbf{b}(T_k)$ across CA--CA and BS--CU links,
subject to: (i) windowed algebraic connectivity among CAs; (ii) minimum rate requirement for each active CA–CA link and minimum uplink and downlink rate requirements for CUs; and (iii) a per-slot total bandwidth budget. The problem is formally stated as follows:
\begin{align}
\min_{\substack{
    a_{ij}(T_k)_{i<j},\\
    \mathbf b(T_k)
}}
\quad
& {\displaystyle  B^{\mathrm{total}}(T_k)}
\label{eq:optimization} \\[4pt]
\text{s.t.} \qquad 
&\text{C}_{1}: \quad \lambda_2\!\Bigg(\sum_{\ell=\max\{k_0,k-\tilde K+1\}}^{k} \mathbf L(T_\ell)\Bigg)
\ge
\underline{\lambda}_{\tilde K},  \nonumber \\
&\text{C}_{2}: \quad
    a_{ij}(T_k) \in \{0,1\},
    \quad \forall\, i<j,\; k, \nonumber\\[4pt]
&\text{C}_{3}: \quad
    R_{ij}(T_k) \ge R_{\min}^{\mathrm{CA}} \, a_{ij}(T_k),
    \forall\, i<j,\; k, \nonumber\\[4pt]
&\text{C}_{4}: \quad \sum_{\substack{i,j\in\mathcal{N}_c\\ i<j}} b_{ij}(T_k)
\;+\;
\sum_{u\in\mathcal{N}_m} b_{u}(T_k)
\;\le\; 1, \nonumber\\
& \qquad \qquad \forall k, \nonumber\\%[4pt]
&\text{C}_{5}: \quad
    b_{ij}(T_k) \in [0,1],
    \quad \forall\, i<j,\; k, \nonumber\\[4pt]
&\text{C}_{6}: \quad
    b_{u}(T_k) \in [0,1], \quad \forall\, u, k, \nonumber\\[4pt]
&\text{C}_{7}: \quad
    %\sum_{k=k_0}^{k_0+K-1} 
    R_u^{\mathrm{DL}}(T_k) \ge R_u^{\mathrm{DL,min}},\nonumber  \quad \forall\, u,\; k, \nonumber\\[4pt]
&\text{C}_{8}: \quad
   %\sum_{k=k_0}^{k_0+K-1}
   R_u^{\mathrm{UL}}(T_k) \ge R_u^{\mathrm{UL,min}},\nonumber \quad \forall\, u,\; k. \nonumber
\end{align}
In problem~\eqref{eq:optimization}, the per-slot
bandwidth-allocation vector is
\[
\mathbf{b}(T_k)
\triangleq
\big(
\{b_{ij}(T_k)\}_{i<j},\;
\{b_{u}(T_k)\}_{u}
\big),
\]
and the total bandwidth consumed in slot $T_k$ is
\begin{align}
B^{\mathrm{total}}(T_k)
&=
\sum_{\substack{i,j \in \mathcal{N}_c\\ i<j}}
\omega\, b_{ij}(T_k)
+
\sum_{u \in \mathcal{N}_m}
\omega\, b_{u}(T_k),
\label{eq:Btotal}
\end{align}
The constraints are interpreted as follows.
C$_1$ enforces the windowed algebraic-connectivity condition
over every $\tilde{K}$ consecutive slots, guaranteeing
consensus in the MACS.
C$_2$ enforces the binary nature of the topology decisions.
C$_3$ guarantees CA--CA information exchange at or above
$R_{\min}^{\mathrm{CA}}$, which by \textbf{Lemma~\ref{lem:Rmin-CA}}
ensures the transmission delay on every active link remains
within the admissible margin $\tau^{*}(T_k)$.
C$_4$ is the per-slot bandwidth budget.
C$_5$ and C$_6$ restrict the bandwidth allocations to valid
normalized fractions in $[0,1]$.
C$_7$ and C$_8$ enforce downlink and uplink QoS for each
CU $u\in\mathcal{N}_m$ via minimum rate requirements
$R_u^{\mathrm{DL,min}}$ and $R_u^{\mathrm{UL,min}}$,
respectively.
The decision variables are the adjacency entries
$a_{ij}(T_k)$, indicating whether agents $i$ and $j$
exchange state information in slot $T_k$, and the
bandwidth allocation vector $\mathbf{b}(T_k)$.
The optimizer selects $a_{ij}(T_k)$ and $\mathbf{b}(T_k)$
to minimize the bandwidth consumption subject to the MACS
consensus and connectivity constraints and the QoS of CA and cellular
users guarantees. Note that, C$_1$ is a sliding-window constraint enforcing the windowed
algebraic connectivity required for consensus over every
$\tilde{K}$ consecutive slots.
As discussed in \textbf{Remark 2}, a smaller
$\tilde{K}$ implies stronger per-slot connectivity and
faster consensus convergence.
For example, $\tilde{K}=1$ requires the D2D graph
$\mathcal{G}(T_k)$ to be connected at every individual
slot $T_k$, allowing state information to propagate across
all agents within a single slot.

%\begin{remark}[Delay-free consensus dynamics]
%\label{rem:delay-free}
%The minimum rate constraint in $C_3$ enforces $R_{i\to j}(T_k)\geq R_{\min}^{\mathrm{CA}}$ on every active link, bounding the one-way packet transmission delay by
%
%\begin{equation*}
%\tau_{ij}(T_k) \;\leq\; \tau_{\max}\;=\; \frac{L_{\mathrm{pkt}}}{R_{\min}^{\mathrm{CA}}},
%\label{eq:tau-max}
%\end{equation*}
%
%where $L_{\mathrm{pkt}}$ [bits] is the control-packet size. Following \cite{olfati2004consensus}, consensus stability is preserved provided $\tau_{ij}(T_k) < \tau^{*}$, where $\tau^{*} = \pi/\bigl(2\lambda_{\max}(\mathbf{L}(T_k))\bigr)$ is the admissible delay margin determined by the largest eigenvalue of the graph Laplacian. Since constraint $C_1$ guarantees $\lambda_2(\mathbf{L}(T_k)) \geq \underline{\lambda}_{\tilde{K}} > 0$, the Laplacian eigenvalues are bounded and $\tau^{*}$ remains strictly positive throughout the operation. By selecting $R_{\min}^{\mathrm{CA}}$ to ensure $\tau_{\max} < \tau^{*}$, the delay-free dynamic model in~\eqref{eq:local-dyn} provides an accurate representation of the delayed dynamics, and communication delay is implicitly accounted for through the rate constraint~\eqref{rate_con}. Therefore,  although D2D communication introduces a finite transmission delay, the consensus dynamics in~\eqref{eq:stacked-dyn} adopt a delay-free model.
%\end{remark}
%In general, mixed-integer nonconvex problems are NP-hard and may require exploring prohibitively large search trees. To handle this difficulty, we propose a solution approach in the next section with polynomial complexity.

\section{Proposed Solution Approach}\label{sec:solution}
The optimization problem~\eqref{eq:optimization} is a mixed-integer nonlinear program (MINP). In general, MINPs are NP-hard, and solving them may require exploring prohibitively large search trees. However, a structural property of the optimal solution to~\eqref{eq:optimization} provides important theoretical insight into the interaction between the communication and control subsystems while enabling a computationally efficient solution approach. In this section, we present the details of the proposed solution approach

%%%%%%%%%%%%%%%%%%%%%%%%
%%%%%%%%%%%%%%
%The optimization problem~\eqref{eq:optimization} is a mixed-integer nonlinear program (MINP). In general, MINPs are NP-hard, and solving them may require exploring prohibitively large search trees. However, a structural property of the optimal solution to~\eqref{eq:optimization} enables a computationally more efficient solution approach. Specifically, \textbf{Lemma~\ref{lem:adjacency_minimal}} below establishes that any optimal solution must use the \emph{minimum} number of active D2D edges satisfying the windowed connectivity constraint~C$_1$. This structural property enables a two-stage solution in which topology selection and bandwidth allocation are solved sequentially. Both stages share the same per-slot channel state information: the topology selection stage uses channel quality as edge weights to prefer links that require less bandwidth, and the bandwidth stage allocates the minimum resources for the selected links based on their SNRs. Both stages are solvable in polynomial time, yielding an overall approach with polynomial per-slot complexity, as detailed in the next section.

\subsection{Structural Properties and Solution Approach}

Although the optimization problem in~\eqref{eq:optimization} is a MINP, it has structural property formalized in \textbf{Lemma~\ref{lem:adjacency_minimal}} below. This structural property enables us to design solution methods with substantially lower complexity. Specifically, \textbf{Lemma~\ref{lem:adjacency_minimal}} establishes that any optimal solution of~\eqref{eq:optimization} must use the \emph{minimum} number of active D2D edges satisfying the windowed connectivity constraint~C$_1$. This structural property reveals that, although the communication topology and wireless resource allocation are jointly optimized, the coupled co-design problem admits an optimal two-stage sequential algorithm. Consequently, the topology selection and bandwidth allocation problems can be solved in a two-stage sequential algorithm without loss of optimality. Both stages share the same per-slot channel state information: the topology selection stage uses channel quality as edge weights to prefer links that require less bandwidth, whereas the bandwidth allocation stage assigns the minimum communication resources required for both CUs and CAs (the selected links) according to their SNRs. 

 %Specifically, we first define a set of edge-minimal graphs
%\begin{equation}
%\mathcal{S}_{\min}\ \triangleq\ \underset{\mathcal{A}:\ \text{C}_1\ \text{holds}}{\arg\min}\ \sum_{k=k_0}^{k_0+K-1}\ \sum_{i<j} a_{ij}(T_k), 
%\end{equation}
%each satisfying C\(_1\). We then solve the problem over \(\mathcal{S}_{\min}\) via an alternating scheme: (i) select \(\mathcal{A}\in\mathcal{S}_{\min}\), (ii) solve the resulting convex bandwidth-allocation subproblem with \(\mathcal{A}\) fixed, (iii) optionally update \(\mathcal{A}\) within \(\mathcal{S}_{\min}\) using the current allocation and repeat. This reduces the original MINLP to a finite collection of convex programs, yielding near-optimal solutions with far lower computational burden. In what follows, to justify the approach, we present a lemma showing that any optimal adjacency \(\mathcal{A}\) must satisfy \(\mathcal{A}\in\mathcal{S}_{\min}\).
\begin{lemma}\label{lem:adjacency_minimal}
Let
\begin{equation}
\mathcal{S}_{\min}
\;\triangleq\;
\underset{\{a_{ij}(T_k)\}_{i<j}:\;\mathrm{C}_1\;\text{holds}}{\arg\min}
\;\sum_{i<j} a_{ij}(T_k)
\end{equation}
be the set of adjacency configurations for slot $T_k$ that satisfy the windowed algebraic-connectivity
constraint~C$_1$ and minimize the number of active edges
in the current slot.
If $H=\!\left(\{a_{ij}(T_k)\}_{i<j},\,\mathbf{b}(T_k)\right)$
is an optimal solution of~\eqref{eq:optimization},
then $\{a_{ij}(T_k)\}_{i<j}\in\mathcal{S}_{\min}$.
\end{lemma}

\begin{proof}
Let us prove it by contradiction.
Note that the Laplacians $\mathbf{L}(T_\ell)$ for
$\ell < k$ are fixed by past decisions; only
$\{a_{ij}(T_k)\}_{i<j}$ is the current decision variable.
Assume that optimal solution
$H^\star=\!\left(\{a^\star_{ij}(T_k)\}_{i<j},\,\mathbf{b}^\star(T_k)\right)$
has $\{a^\star_{ij}(T_k)\}_{i<j}\notin\mathcal{S}_{\min}$.
Then there exists $\{\widetilde{a}_{ij}(T_k)\}_{i<j}
\in\mathcal{S}_{\min}$ such that
$\widetilde{a}_{ij}(T_k)\leq a^\star_{ij}(T_k)$
for all $i<j$, with strictly fewer active edges,
and C$_1$ still holds.
Construct $\widetilde{H}=\!\left(\{\widetilde{a}_{ij}(T_k)\}_{i<j},
\,\widetilde{\mathbf{b}}(T_k)\right)$ by setting
$\widetilde{b}_{ij}(T_k)=b^\star_{ij}(T_k)$ if
$\widetilde{a}_{ij}(T_k)=1$, and
$\widetilde{b}_{ij}(T_k)=0$ otherwise;
keep $\widetilde{b}_u(T_k)=b^\star_u(T_k)$ for all
$u\in\mathcal{N}_m$.
Solution $\widetilde{H}$ satisfies:
C$_1$ by definition of $\mathcal{S}_{\min}$;
C$_2$ trivially;
C$_3$ since for $\widetilde{a}_{ij}=0$ the right-hand
side is zero, and for $\widetilde{a}_{ij}=1$ the
original feasible allocation is retained;
C$_4$--C$_6$ since no allocation is increased;
C$_7$--C$_8$ since cellular allocations are unchanged.
Therefore, $B^{\mathrm{total}}(T_k)$ decreases
strictly: at least one edge active in
$\{a^\star_{ij}(T_k)\}_{i<j}$ is inactive in
$\{\widetilde{a}_{ij}(T_k)\}_{i<j}$, and by C$_3$
that edge carries positive bandwidth in slot $T_k$,
which is set to zero in $\widetilde{H}$.
This contradicts the optimality of $H^\star$.
\end{proof}
Leveraging \textbf{Lemma~\ref{lem:adjacency_minimal}}, no optimal solution employs an adjacency
$\{a_{ij}(T_k)\}_{i<j}\notin\mathcal{S}_{\min}$.
Therefore, \textbf{without loss of optimality}, the feasible search space can be reduced to
$\mathcal{S}_{\min}$. Accordingly, the original optimization problem can be solved in two sequential stages:
(i) select an adjacency
$\{a_{ij}(T_k)\}_{i<j}\in\mathcal{S}_{\min}$;
(ii) solve the resulting convex bandwidth-allocation problem for the selected adjacency.
The two subproblems are:
\begin{align}
\min_{\{a_{ij}(T_k)\}_{i<j}}\quad
&\sum_{i<j} a_{ij}(T_k)
\nonumber\\[3pt]
\text{s.t.}\quad
&\text{C}_{1}:\lambda_2\!\Bigg(
  \sum_{\ell=\max\{k_0,\,k-\tilde{K}+1\}}^{k}
  \mathbf{L}(T_\ell)
\Bigg)
\;\ge\;
\underline{\lambda}_{\tilde{K}},\nonumber\\
&\text{C}_{2}:\quad
  a_{ij}(T_k)\in\{0,1\},\quad\forall\,i<j,
\nonumber\\[3pt]
\label{eq:sub_1}
\end{align}
and
\begin{align}
\min_{\mathbf{b}(T_k)}\quad
&B^{\mathrm{total}}(T_k)
\label{eq:sub_2}\\[4pt]
\text{s.t.}\quad
&\text{C}_{3}^{\star}:\quad
  R_{ij}(T_k)\ge R_{\min}^{\mathrm{CA}},
  \quad\forall\,i<j\ \text{with}\ a_{ij}(T_k)=1,
\nonumber\\[4pt]
&\text{C}_{4}\text{--}\text{C}_{8}.\nonumber
\end{align}
Since multiple adjacency configurations may satisfy
C$_1$ with the same minimum edge count, the choice
among them affects the bandwidth cost in
subproblem~\eqref{eq:sub_2}.
To bias the selection toward topologies that are
cheaper for bandwidth allocation, we introduce
channel-aware edge weights.
Specifically, from \eqref{eq:SNR-i-to-j}--\eqref{eq:rate-j-to-i},
the bandwidth required to meet $R_{\min}^{\mathrm{CA}}$
on edge $(i,j)$ in slot $T_k$ decreases monotonically
with the channel SNR $c_{ij}(T_k)=P_{\mathrm{CA}}|h_{ij}|^2/N_0$.
We therefore set the edge weight as
\begin{equation}
w_{ij}(T_k)
\;\propto\;
\frac{1}{\log_2\!\left(1+c_{ij}(T_k)\right)},
\label{eq:weights}
\end{equation}
and replace the unweighted objective in~\eqref{eq:sub_1}
with the bandwidth-aware surrogate:
\begin{align}
\min_{\{a_{ij}(T_k)\}_{i<j}}\quad
&\sum_{i<j} w_{ij}(T_k)\,a_{ij}(T_k)
\nonumber\\[3pt]
\text{s.t.}\quad
&\text{C}_{2}:\quad
  a_{ij}(T_k)\in\{0,1\},\quad\forall\,i<j,
\nonumber\\[3pt]
&\lambda_2\!\Bigg(
  \sum_{\ell=\max\{k_0,\,k-\tilde{K}+1\}}^{k}
  \mathbf{L}(T_\ell)
  \Bigg)
\;\ge\;
\underline{\lambda}_{\tilde{K}}.
\label{eq:sub_1_W}
\end{align}
This focuses the topology selection on edges that
can meet C$_3^\star$ with less bandwidth in
subproblem~\eqref{eq:sub_2}, thereby reducing
$B^{\mathrm{total}}(T_k)$ at the second stage.
%which prioritizes activating edges that can meet C\(_3^\star\) (in subproblem \eqref{eq:sub_2}) with less bandwidth. This focuses the search on those minimal-edge graphs most likely to yield a smaller \(\sum_k B^{\mathrm{total}}(T_k)\) in subproblem \eqref{eq:sub_2}.

In what follows, we first present a solution for the minimal-edge graph subproblem \eqref{eq:sub_1_W} to select an adjacency \(\mathcal{A}\in\mathcal{S}_{\min}\) using a spectral linear matrix inequality (LMI) formulation of the windowed algebraic-connectivity requirement. We then address the bandwidth-allocation subproblem \eqref{eq:sub_2}, which, for fixed \(\{a_{ij}(T_k)\}_{i<j}\), reduces to a convex optimization problem.

\subsection{Minimal-Edge Graph Subproblem via Spectral LMI}\label{sec:solution:A}
%%%%%%%%%%%%%%%%%
%A remaining challenge is that searching over all members of $\mathcal{S}_{\min}$ is still computationally expensive: the edge selection per slot involves $\binom{N_c}{2}$ binary variables, windows overlap, and the number of feasible topologies satisfying C$_1$ with exactly the minimum edge count can grow exponentially with $N_c$.
%%%%%%%%%%%%%%%
To solve~\eqref{eq:sub_1_W}, we express C$_1$ via its linear matrix inequality (LMI) representation so that the constraint is affine in the adjacency variables $\{a_{ij}(T_k)\}_{i<j}$ through the windowed Laplacian.
We then handle the resulting mixed-integer semidefinite program (MISDP) using a relax-and-round approach: relax $a_{ij}(T_k)\in\{0,1\}$ to continuous variables $\alpha_{ij}(T_k)\in[0,1]$, solve the convex SDP, then round and prune while re-checking the LMI to obtain a feasible binary solution.

\begin{remark}[Semidefinite form of the joint connectivity
condition \cite{ghosh2006growing}]
C$_1$ is equivalent to the LMI
\begin{equation}
\sum_{\ell=\max\{k_0,\,k-\tilde{K}+1\}}^{k}
\mathbf{L}(T_\ell)
\;-\;
\underline{\lambda}_{\tilde{K}}
\Big(\mathbf{I}_{N_c}
     -\tfrac{1}{N_c}\mathbf{1}\mathbf{1}^\top\Big)
\succeq 0.
\label{eq:windowed-sdp-thm_}
\end{equation}
Representing C$_1$ as~\eqref{eq:windowed-sdp-thm_} makes
the constraint affine in
$\{a_{ij}(T_\ell)\}_{\ell=\max\{k_0,k-\tilde{K}+1\}}^{k}$,
allowing convex relaxation once the binary constraints
are relaxed.
\end{remark}

Because $\mathbf{L}(T_\ell)$ is affine in the edge variables, \eqref{eq:windowed-sdp-thm_} is convex in
continuous variables. We therefore relax $a_{ij}(T_k)\in\{0,1\}$ to $\alpha_{ij}(T_k)\in[0,1]$
and rewrite~\eqref{eq:sub_1_W} as the SDP:
\begin{align}
\min_{\{\alpha_{ij}(T_k)\}_{i<j}}\quad
&\sum_{i<j} w_{ij}(T_k)\,\alpha_{ij}(T_k)
\label{eq:sdp-relaxation}\\
\text{s.t.}\quad
&\eqref{eq:windowed-sdp-thm_},\quad
 0\leq\alpha_{ij}(T_k)\leq 1,\quad\forall\,i<j.\nonumber
\end{align}

Post-processing yields binary variables via:
\begin{enumerate}
\item \textbf{Rounding:} For each slot $T_k$,
      either (a) threshold edges where
      $\alpha_{ij}(T_k)\geq\tau$ (e.g., $\tau=0.5$),
      or (b) retain the top-$\kappa$ edges with the
      largest $\alpha_{ij}(T_k)$,
      to obtain $a_{ij}(T_k)\in\{0,1\}$.
\item \textbf{Prune and certify:} Greedily remove
      selected edges in decreasing order of $w_{ij}(T_k)$
      while~\eqref{eq:windowed-sdp-thm_} remains
      satisfied; if removal violates
      \eqref{eq:windowed-sdp-thm_}, the edge is restored.
      This produces a minimal (or near-minimal) feasible
      solution.
\end{enumerate}
\subsection{Bandwidth Allocation Subproblem}\label{sec:solution:B}
Given a feasible set of adjacency variables $\{a_{ij}(T_k)\}_{i<j}$, subproblem~\eqref{eq:sub_2} is convex with respect to $\mathbf{b}(T_k)$. Specifically, the objective function $B^{\mathrm{total}}(T_k)$ is linear in $\mathbf{b}(T_k)$. Moreover, all rate functions in \eqref{eq:rate-i-to-j}, \eqref{eq:rate-j-to-i}, \eqref{eq:rate-UL}, and \eqref{eq:rate-DL} are concave with respect to $\mathbf{b}(T_k)$, since the second derivative of the SNR-based term is strictly negative. Consequently, the QoS constraints of CAs in C$_3^\star$ and those of CUs in C$_7$ and C$_8$ are convex. In addition, constraints C$_4$--C$_6$ are affine in $\mathbf{b}(T_k)$. Thus, bandwidth allocation subproblem  \eqref{eq:sub_2} is convex, and its global optimal solution can be obtained using standard convex optimization tools (e.g., CVX/CVXPy with MOSEK).

\subsection{Overall Procedure of the Proposed Approach}
In this subsection, we summarize the overall procedure of the proposed approach in \textbf{Algorithms 1–4}.
\begin{algorithm}[t]
\footnotesize
\caption{Consensus-Aware Co-Design at Slot $T_k$}
\label{alg:main}
\DontPrintSemicolon
\KwInput{%
  $\tilde{K}$, $\underline{\lambda}_{\tilde{K}}$,
  $\kappa$, $N_c$, $N_m$;
  channels $c_{ij}(T_k)$, $c_u^{\mathrm{UL}}(T_k)$,
  $c_u^{\mathrm{DL}}(T_k)$;
  rates $R_{\min}^{\mathrm{CA}}$,
  $R_u^{\mathrm{UL,min}}$, $R_u^{\mathrm{DL,min}}$;
  past Laplacians
  $\mathbf{L}_{<k}\!\coloneqq\!
   \{\mathbf{L}(T_\ell)\}
   _{\ell=\max\{k_0,k-\tilde{K}+1\}}^{k-1}$
}
\KwOutput{%
  $\{a_{ij}(T_k)\}_{i<j}$,\quad
  $\mathbf{b}(T_k)$}

\tcp{Step 1: Channel-aware edge weights}
\ForEach{$(i,j):\,i<j$}{
  $w_{ij}(T_k)
    \gets [\log_2(1+c_{ij}(T_k))]^{-1}$\;
}

\tcp{Step 2: SDP relaxation}
$\{\alpha_{ij}(T_k)\}_{i<j}
  \gets \textsc{SDP\_Relax}\bigl(
    \{w_{ij}(T_k)\},\,
    \underline{\lambda}_{\tilde{K}},\,
    \tilde{K},\,N_c,\,
    \mathbf{L}_{<k}\bigr)$
\tcp*{Alg.~\ref{alg:sdp}}

\tcp{Step 3: Round and certify C$_1$}
$\{a_{ij}(T_k)\}_{i<j}
  \gets \textsc{RoundAndPrune}\bigl(
    \{\alpha_{ij}(T_k)\},\,
    \{w_{ij}(T_k)\},
    \kappa,\,
    \underline{\lambda}_{\tilde{K}},\,
    \tilde{K},\,
    \mathbf{L}_{<k}\bigr)$\\
\tcp*{Alg.~\ref{alg:round}}

\tcp{Step 4: Convex bandwidth allocation}
$\mathbf{b}(T_k)
  \gets \textsc{BandwidthAlloc}\bigl(
    \{a_{ij}(T_k)\},\,
    c_{ij}(T_k),\,
    c_u^{\mathrm{UL}}(T_k),\,
    c_u^{\mathrm{DL}}(T_k),
    R_{\min}^{\mathrm{CA}},\,
    R_u^{\mathrm{UL,min}},\,
    R_u^{\mathrm{DL,min}}\bigr)$\\
\tcp*{Alg.~\ref{alg:bw}}

\KwRet $\{a_{ij}(T_k)\}_{i<j},\ \mathbf{b}(T_k)$
\end{algorithm}

% In the preamble (once)
% \SetKwFunction{Lapl}{Laplacian}

\begin{algorithm}[t]
\footnotesize
\caption{SDP Relaxation for Windowed Spectral Connectivity}
\label{alg:sdp}
\DontPrintSemicolon
\setcounter{AlgoLine}{0}
\KwInput{%
\(
\left\{
\begin{aligned}
&\{w_{ij}(T_k)\}_{i<j}
\qquad \text{(channel-aware edge weights)},\\
&\underline{\lambda}_{\tilde{K}},\ \tilde{K},\ N_c,\\
&\{\mathbf{L}(T_\ell)\}_{\ell=\max\{k_0,\,k-\tilde{K}+1\}}^{k-1}
\quad\text{(past Laplacians, fixed)}
\end{aligned}
\right.
\)
}
\KwOutput{$\{\alpha_{ij}(T_k)\}_{i<j}\in[0,1]$ \quad (relaxed edge weights)}

% Decision variable for current slot only
Define symmetric variable
$\boldsymbol{\alpha}(T_k)\in[0,1]^{N_c\times N_c}$
with $\alpha_{ii}(T_k)=0\ \forall\,i$
and $\alpha_{ij}(T_k)=\alpha_{ji}(T_k)\ \forall\,i<j$\;

% Objective
$\mathrm{obj} \gets
  \displaystyle\sum_{i<j} w_{ij}(T_k)\,\alpha_{ij}(T_k)$\;

% Projection matrix onto consensus subspace
$P \gets I_{N_c}-\tfrac{1}{N_c}\mathbf{1}\mathbf{1}^\top$\;

% Build windowed Laplacian: past slots fixed, current slot variable
$L_{\mathrm{agg}} \gets
  \displaystyle
  \sum_{\ell=\max\{k_0,\,k-\tilde{K}+1\}}^{k-1}
  \mathbf{L}(T_\ell)$
\tcp*{fixed past Laplacians}

$L_{\mathrm{agg}} \gets L_{\mathrm{agg}}
  + \textsc{Laplacian}\!\bigl(\boldsymbol{\alpha}(T_k)\bigr)$
\tcp*{current slot — affine in $\boldsymbol{\alpha}(T_k)$}

% LMI constraint (single, for current window)
Impose:\quad
$L_{\mathrm{agg}} - \underline{\lambda}_{\tilde{K}}\,P
  \succeq 0$\;

% Solve SDP
Solve $\min\ \mathrm{obj}$ subject to all constraints
(e.g., via CVX/SDPT3)\;

\KwRet $\{\alpha_{ij}(T_k)\}_{i<j}$
\end{algorithm}

\begin{algorithm}[t]
\footnotesize
\caption{\textsc{RoundAndPrune}: Binary Rounding
  and Minimal-Edge Certification}
\label{alg:round}
\DontPrintSemicolon
\KwInput{%
  $\{\alpha_{ij}(T_k)\}_{i<j}\in[0,1]$
  (SDP relaxation solution);
  $\{w_{ij}(T_k)\}_{i<j}$ (edge weights);
  $\kappa$ (max edges);
  $\underline{\lambda}_{\tilde{K}}$,
  $\tilde{K}$,
  $\mathbf{L}_{<k}$ (past Laplacians)
}
\KwOutput{$\{a_{ij}(T_k)\}_{i<j}\in\{0,1\}$
  (minimal binary adjacency satisfying C$_1$)}

\tcp{--- Phase 1: Top-$\kappa$ Rounding ---}
$a_{ij}(T_k)\gets 0\quad\forall\,i<j$\;
$\mathcal{E}\gets\{(i,j,\alpha_{ij}(T_k)):i<j\}$
sorted descending by $\alpha_{ij}(T_k)$\;
$m\gets\min\{\kappa,|\mathcal{E}|\}$\;
\For{$t\gets 1$ \KwTo $m$}{
  $(i,j,\cdot)\gets\mathcal{E}[t]$;\quad
  $a_{ij}(T_k)\gets 1$;\quad
  $a_{ji}(T_k)\gets 1$\;
}

\tcp{--- Phase 2: Prune and Certify C$_1$ ---}
$\mathcal{E}_{\mathrm{act}}\gets
  \{(i,j):a_{ij}(T_k)=1,\;i<j\}$
  sorted descending by $w_{ij}(T_k)$\;
\ForEach{$(i,j)\in\mathcal{E}_{\mathrm{act}}$}{
  $a_{ij}(T_k)\gets 0$;\;
  $a_{ji}(T_k)\gets 0$
  \tcp*{tentatively remove}
  $L_{\mathrm{agg}}\gets
    \mathbf{L}_{<k}
    +\textsc{Laplacian}\!\bigl(
      \{a_{ij}(T_k)\}\bigr)$\;
  \eIf{$L_{\mathrm{agg}}
       -\underline{\lambda}_{\tilde{K}}
       (I_{N_c}
       -\tfrac{1}{N_c}\mathbf{1}\mathbf{1}^\top)
       \succeq 0$}{
    \tcp{C$_1$ holds — keep removal}
  }{
    $a_{ij}(T_k)\gets 1$;\quad
    $a_{ji}(T_k)\gets 1$
    \tcp*{restore edge}
  }
}
\KwRet $\{a_{ij}(T_k)\}_{i<j}$
\end{algorithm}

\begin{algorithm}[t]
\footnotesize
\caption{\textsc{BandwidthAllocation}:
  Convex Program for Fixed Topology}
\label{alg:bw}
\DontPrintSemicolon
\KwInput{%
\(
\left\{
\begin{array}{l}
\{a_{ij}(T_k)\}_{i<j}
,\\[1mm]
R_{\min}^{\mathrm{CA}},\quad
\{R_u^{\mathrm{UL,min}}\}_u,\quad
\{R_u^{\mathrm{DL,min}}\}_u
\end{array}
\right.
\)
}

\KwOutput{$\mathbf{b}(T_k)
  =\bigl(\{b_{ij}(T_k)\}_{i<j},\,
         \{b_u(T_k)\}_u\bigr)$}

% Variables
\textbf{Variables.} Define nonnegative bandwidths $b_{ij}[k]$ for active D2D edges and $b_u[k]$ for CUs.\;

\textbf{Objective.} Minimize total bandwidth in \eqref{eq:sub_2}.\;

\textbf{Constraints.} Enforce C\(_3^\star\), C\(_4\)–C\(_8\) as stated in the problem.\;

\textbf{Solve.} Use a convex solver (e.g., CVX).\;

\textbf{}\KwRet $\mathbf{b}(T_k)$
\end{algorithm}
\subsection{Analysis of Computational Complexity}
% =========================================================
%  Complexity Analysis — Corrected Version
%  Ready-to-paste into Overleaf
% =========================================================

The overall per-slot computational complexity is
dominated by subproblem~\eqref{eq:sub_1_W}.

\textbf{SDP relaxation (Algorithm~\ref{alg:sdp}).}
The SDP in~\eqref{eq:sdp-relaxation} involves
$n=O(N_c^2)$ decision variables
$\{\alpha_{ij}(T_k)\}_{i<j}$ and one LMI constraint
of size $N_c\times N_c$.
Using interior-point methods, the solver requires
$O(\sqrt{N_c})$ iterations, each costing $O(N_c^6)$
operations to form and factorize the $n\times n$
Schur complement matrix, yielding a per-slot SDP
complexity of $O(N_c^{6.5})$.

\textbf{Rounding and pruning (\textbf{Algorithm~\ref{alg:round})}.}
Top-$\kappa$ rounding sorts $O(N_c^2)$ edges in
$O(N_c^2\log N_c)$.
Pruning attempts at most $\kappa$ edge removals, each
requiring a Laplacian eigenvalue check costing
$O(N_c^3)$, giving $O(\kappa N_c^3)$ for the
pruning phase.

\textbf{Bandwidth allocation (\textbf{Algorithm~\ref{alg:bw})}.}
Subproblem~\eqref{eq:sub_2} is a convex LP with
$O(N_c^2+N_m)$ variables and $O(N_c^2+N_m)$
constraints.
For an LP with $n$ variables, interior-point methods
require $O(\sqrt{n})$ iterations each costing $O(n^2)$,
giving $O(n^{2.5})$ total.
The per-slot LP complexity is therefore
$O\!\bigl((N_c^2+N_m)^{2.5}\bigr)$.

\textbf{Total per-slot complexity.}
The dominant term is the SDP relaxation, giving
\begin{equation}
O\!\Bigl(
  N_c^{6.5}
  + \kappa\, N_c^3
  + (N_c^2+N_m)^{2.5}
\Bigr),
\label{eq:complexity}
\end{equation}
which is polynomial in $N_c$ and $N_m$.

%The proposed method achieves polynomial complexity \eqref{eq:complexity} by decomposing the NP-hard mixed-integer problem into three tractable stages. First, Lemma~\ref{lem:adjacency_minimal} restricts the search to minimum-edge topologies satisfying C$_1$, eliminating all suboptimal adjacency configurations without loss of optimality. Second, the binary constraint is relaxed to a convex SDP, solvable in polynomial time via interior-point methods; rounding and pruning then recover a feasible binary topology in $O(\kappa N_c^3)$. Third, given the fixed topology, the bandwidth allocation in Algorithm~\ref{alg:bw} is solved to \emph{global optimality} via a convex LP. The only source of suboptimality is the SDP relaxation gap in topology selection. The SDP provides a computable lower bound on the optimal weighted edge cost; the rounded solution provides an upper bound; and their ratio quantifies the optimality gap. By Lemma~\ref{lem:adjacency_minimal}, the pruning step guarantees the final topology is minimal in edge count, so the suboptimality affects only \emph{which} minimum-edge topology is selected, not \emph{how many} edges it contains.
\section{Performance Evaluation}\label{sec:simulation}

\subsection{Simulated Multi-agent Control and Communication Model}\label{sec:simulation_system}
We evaluate the proposed co-design approach using a system model comprising 20 CAs and 40 CUs. Specifically, the 20 CAs correspond to the control areas of a 20-area load--frequency control (LFC) system, in which distributed energy resources (DERs) introduce stochastic load disturbances that require distributed coordination among the control areas. The 20 CAs exchange their local area control error (ACE) integral states through D2D wireless links to achieve consensus. Consensus of the ACE integral states enables coordinated secondary frequency regulation and maintains the tie-line power balance among all control areas. The D2D links share the licensed spectrum with $N_m$ mobile CUs, each requiring a minimum uplink and downlink rate. System performance is assessed in terms of:
(i)~ACE consensus disagreement
     $\|\tilde{\mathbf{X}}(t)\|$,
(ii)~frequency deviations $\Delta f_i(t)$,
(iii)~windowed algebraic connectivity
      $\lambda_2\!\left(\sum_\ell \mathbf{L}(T_\ell)\right)$, and
(iv)~normalized total bandwidth consumption
     $B^{\mathrm{total}}(T_k)/\omega$.
%We evaluate the proposed co-design approach for a distributed power system cooperating through wireless communication links. In this system, a 20-area load–frequency control system with distributed energy resources (DERs) shares the wireless spectrum with CUs. Each control area is represented by one CA. The agents regulate local frequency deviations in the distributed power system and exchange coordination information via wireless D2D  links. Specifically, the multi-agent control objective is to achieve secondary frequency control and tie-line power balance while sharing bandwidth with mobile CUs that require a minimum QoS. In this system, performances are assessed in terms of frequency deviations, area control error (ACE), consensus disagreement, and bandwidth consumption. %connectivity ($\lambda_2$),
%\subsection{Distributed Power System}
%\label{sec:simulation_system}

Each control area is represented by one CA, forming a 20-node MACS. Each CA exchanges frequency deviation (\(\Delta f_i(t)\)) and ACE ($\xi_i(t)$) with neighboring agents. The dynamics of each agent $i$ is \cite{singh2017distributed}:
\begin{align}
\Delta\dot f_i(t) &= -\frac{D_i}{M_i}\,\Delta f_i(t)
+ \frac{1}{M_i}\,\Delta P_{t,i}(t) \nonumber \\
&- \frac{1}{M_i}\,\Delta P^{i}_{tie}(t)
- \frac{1}{M_i}\,\Delta P_{L_i}(t),\label{eq:area-dynamics} \\
\Delta\dot P_{t_i}(t) &= -\frac{1}{T_{t_i}}\,\Delta P_{t_i}(t)
+ \frac{1}{T_{t_i}}\,\Delta P_{g_i}(t),\nonumber\\
\Delta\dot P_{g_i}(t) &= - \frac{1}{R_i T_{g_i}}\,\Delta f_i(t)
-\frac{1}{T_{g_i}}\,\Delta P_{g_i}(t)
+ \frac{1}{T_{g_i}}\,\Delta P_{c_i}(t),\nonumber\\
\Delta\dot P^{i}_{tie}(t) &= \sum_{\substack{j=1\\ j\ne i}}^{N_c} 
2\pi\,T_{ij}\big(\Delta f_i (t)- \Delta f_j(t)\big),\nonumber\\
\dot{\xi}_i(t) &=
-\gamma_i\,\Delta f_i(t)
-\sum_{j\in\mathcal{N}_c}
a_{ij}(T_k)\,
\big(\xi_i(t)-\xi_j(t)\big),\nonumber   
\end{align}
\noindent where $\Delta f_i$ is the frequency deviation, $\Delta P_{g_i}$ and $\Delta P_{t_i}$ are the governor and turbine power deviations, respectively, $\Delta P_{L_i}$ is the load disturbance, $\Delta P_{c_i}$ is the secondary control input, and $\Delta P^{i}_{tie}$ represents the tie-line power deviation of agent $i$. The variable ${\xi_i}$ is the integral state associated with the ACE, which accumulates frequency deviations and neighboring discrepancies to achieve secondary frequency restoration and consensus among the CAs. %$\mathbf{L}_T$ captures the physical coupling through tie-lines (appearing in $\Delta {\dot{P}}^{\,i}_{\mathrm{tie}}$), while $\mathbf{L}(T_k)$ captures the information coupling over the slot-$T_k$ D2D graph (appearing in $\dot \xi_i$). The local term $-\gamma_i\Delta f_i$ is embedded in $\mathbf{A}_i^d$ .
All other symbols are defined as: $M_i$ (inertia), $D_i$ (damping), $R_i$ (droop),
$T_{g_i}$ (governor time constant), $T_{t_i}$ (turbine time constant), $T_{ij}$ (synchronizing
coefficient of tie-line $i\!-\!j$), $\Delta P_{L_i}$ (load disturbance).
Let each CA's state be defined as  $\mathbf {x}_i(t)=\begin{bmatrix}\Delta f_i(t)& \Delta P_{t,i}(t) & \Delta P_{g,i}(t) & \Delta P^{i}_{\mathrm{tie}}(t) & \xi_i(t)\end{bmatrix}^{\!\top}\!\in\mathbb{R}^{5}$,
the control input as $\mathbf{u}_i(t)\triangleq \Delta P_{c,i}(t)$ and the exogenous disturbance as $E_i(t)=- \frac{1}{M_i}\,\Delta P_{L_i}(t)$, the dynamics in \eqref{eq:area-dynamics} can be written in state–space form as follows:
\begin{align}
\dot{\mathbf{x}}_i(t)
&= \mathbf{A}_i^d\,\mathbf{x}_i(t) \;+\; \mathbf{B}_i\,\mathbf{u}_i(t) \;+\; \mathbf{E}_i(t) \nonumber\\
&- \sum_{j=1}^{N_c}\! \Big( \mathbf{L}_T(i,j)\,\mathbf{C}_T\,\mathbf{x}_j(t) 
\;+\; \mathbf{L_c}(T_k)(i,j)\mathbf{C_c}\,\mathbf{x}_j(t) \Big),
\label{eq:ss-single}
\end{align}
where
\[
\mathbf{A}_i^d =
\begin{bmatrix}
-\frac{D_i}{M_i} & + \frac{1}{M_i} & 0 & - \frac{1}{M_i} & 0\\
0 & -\dfrac{1}{T_{t_i}} & \dfrac{1}{T_{t_i}} & 0 & 0\\
-\dfrac{1}{R_iT_{g_i}} & 0 & -\dfrac{1}{T_{g_i}} & 0 & 0\\
\displaystyle \sum_{j=1,\,j\ne i}^{N_c} 2\pi T_{ij} & 0 & 0 & 0 & 0\\
-\gamma_i & 0 & 0 & 0 & 0
\end{bmatrix},
\]

\[
\mathbf{B}_i = 
\begin{bmatrix}
0 & 0 & \dfrac{1}{T_{g_i}} & 0 & 0
\end{bmatrix}^{\!\top},
\]
\[
\mathbf{E}_i =
\begin{bmatrix}
-\dfrac{1}{M_i}\Delta P_{L_i}(t) & 0 & 0 & 0 & 0
\end{bmatrix}^{\!\top}.
\]
\noindent This MACS involves two types of Laplacian matrices: one capturing physical coupling ($\mathbf{L}_T$) and another capturing information coupling ($\mathbf{L}_c(T_k)$). The constant matrix $\mathbf{L}_T\in\mathbb{R}^{N_c\times N_c}$ is the tie-line Laplacian associated with the underlying power network. It is constructed using symmetric matrix $[W_T]_{ij}=2\pi T_{ij}$, i.e.,
$\mathbf{L}_T=\mathrm{diag}(W_T\mathbf{1})-W_T$. Likewise,  $\mathbf{L}_c(T_k)\triangleq \mathrm{diag}(\mathbf{A}(T_k)\mathbf{1})-\mathbf{A}(T_k)$ is the communication Laplacian induced by the D2D adjacency matrix $\mathbf{A}(T_k)=[a_{ij}(T_k)]$ in slot $T_k$. Also, $\mathbf{C}_T$ and $\mathbf{C}_c$ are coupling (consensus) gain matrices, where $\mathbf{C}_T$ captures physical tie-line coupling and $\mathbf{C}_c$ captures information coupling. 
\begin{remark}
The agent dynamics in \eqref{eq:ss-single} constitute a structured instance of the general
consensus-coupled multi-agent model in \eqref{eq:stacked-dyn}.
In particular, the interaction terms are decomposed into two components:
(i) a physical coupling induced by power tie-lines (appearing in $\Delta {\dot{P}}^{\,i}_{\mathrm{tie}}$), captured by the constant Laplacian
$\mathbf{L}_T$ with coupling matrix $\mathbf{C}_T$, and
(ii) an information coupling induced by D2D wireless communication (appearing in
$\dot \xi_i$), captured by the
time-varying Laplacian $\mathbf{L}_c(T_k)$ with coupling gain $\mathbf{C}_c$. %Since the physical tie-line topology and parameters are fixed, the Laplacian $\mathbf{L}_T$ is time-invariant. It indirectly influences the communication-topology optimization by shaping the coordination requirements among agents. In contrast, the communication Laplacian $\mathbf{L}_c(T_k)$ is explicitly optimized to satisfy windowed spectral-connectivity constraints while minimizing communication resource usage. 
Both couplings take the standard diffusive form
%$\sum_j \mathbf L(T_k)\mathbf{C}(\mathbf{x}_j-\mathbf{x}_i)$
$-\big(\mathbf{L}(T_k)\otimes \mathbf{C}\big)\mathbf{X}$
 and therefore preserve the consensus structure assumed in \eqref{eq:stacked-dyn}.
\end{remark}
 %Channels and allocations are piecewise constant within each slot; $\mathbf{A}(T_k)$, $\mathbf{L}_c(T_k)$, and the wireless gains are held fixed over \([kT_s,\,(k+1)T_s]\). 
By stacking states $\mathbf{X}\!\triangleq\!\col(\mathbf{x}_1,\dots,\mathbf{x}_{N_c})\in\mathbb{R}^{5\times N_c}$, inputs
$\mathbf{U}\!\triangleq\!\col(\mathbf{u}_1,\dots,\mathbf{u}_{N_c})$, and disturbances
$\mathbf{E}\!\triangleq\!\col(\mathbf{E}_1,\dots,\mathbf{E}_{N_c})$, and by defining
$\mathbf{A}^d \triangleq \blkdiag(\mathbf{A}_1^d,\dots,\mathbf{A}_{N_c}^d)$ ,\qquad
$\mathbf{B} \;\triangleq\; \blkdiag(\mathbf{B}_1,\dots,\mathbf{B}_{N_c})$, overall system dynamics evolves as:
\begin{equation}
\begin{aligned}
\dot{\mathbf{X}}(t)
&=
\Big(\mathbf{A}^d 
\;-\; \mathbf{L}_T\otimes \mathbf{C}_T
\;-\; \mathbf{L}_c(T_k)\otimes \mathbf{C}_c\Big)\mathbf{X}(t) \\[4pt]
&\quad +\;\mathbf{B}\,\mathbf{U}(t)\;+\;\mathbf{E}(t),
\qquad t\in T_k,
\label{eq:ss-stacked}
\end{aligned}
\end{equation}
which is in the standard form \eqref{eq:stacked-dyn}.% used in consensus-coupled multi-agent control.
%We simulate the four–area LFC in MATLAB, the time horizon is divided into slots $T_k$ of duration $0.5~\mathrm{s}$, and inside each slot, the continuous dynamics is integrated with step $10~\mathrm{ms}$. 
The remaining power system parameters are listed in Table~\ref{tab:numvalues}.
%Small step load changes are applied to the area loads $\Delta P_{L_i}$.
\begin{table}[h]
\centering
\caption{Simulation parameters}
\label{tab:numvalues}
\scriptsize
\renewcommand{\arraystretch}{1.3}
\begin{tabular}{@{}l l c@{}}
\toprule
\textbf{Parameter} & \textbf{Description} & \textbf{Value} \\
\midrule
\multicolumn{3}{@{}c}{\textbf{Power system parameters (per area $i$)}}\\\hline
$M_i$ (p.u.$\cdot$s) & Inertia constant & 12 \\
$D_i$ (p.u./Hz) & Damping coefficient & $\{1.5, 1.8, 1.5, 1.8\}$ \\
$R_i$ & Speed droop & 0.05 \\
$T_{g_i}$ (s) & Governor time constant & $\{0.40, 0.35, 0.40, 0.35\}$ \\
$T_{t_i}$ (s) & Turbine time constant & $\{0.17, 0.20, 0.17, 0.20\}$ \\
\midrule

\multicolumn{3}{@{}c}{\textbf{Shared power-system parameters}}\\\hline
$T_{ij}$ (p.u./Hz) & Tie-line synchronizing coefficient & 0.05 \\
$\gamma_i$ & Integral/ACE gain & 0.5 \\
\midrule

\multicolumn{3}{@{}c}{\textbf{Wireless network parameters}}\\\hline
$P_{\mathrm{CA}}$ (W) & CA transmit power & 0.2 \\
$P_u^{\mathrm{UL}}$ (W) & CU uplink transmit power & 0.05 \\
$P_u^{\mathrm{DL}}$ (W) & BS  transmit power to each CU  & 0.5 \\
$\omega$ (MHz)& Total available bandwidth & 10 \\
$N_0$ (W/Hz) & Noise power spectral density & $10^{-17.4}$ \\
$R_{\min}^{\mathrm{CA}}$ (kbps) & Minimum CA–CA D2D rate & 10 \\
$L_{\mathrm{pkt}}$ (bits) & Control-packet
size  & 400 \\
$R_u^{\mathrm{UL,min}}$ (Mbps) & Minimum CU uplink rate & 2.8 \\
$R_u^{\mathrm{DL,min}}$ (Mbps) & Minimum CU downlink rate & 3.0 \\
$\eta$ & Path-loss exponent & 2.5 \\
\bottomrule
\end{tabular}
\end{table}
\begin{figure*}[t]
    \centering
    %==================== Top row ====================%
    \begin{minipage}[t]{0.45\textwidth}
        \centering
        \includegraphics[width=0.85\linewidth]{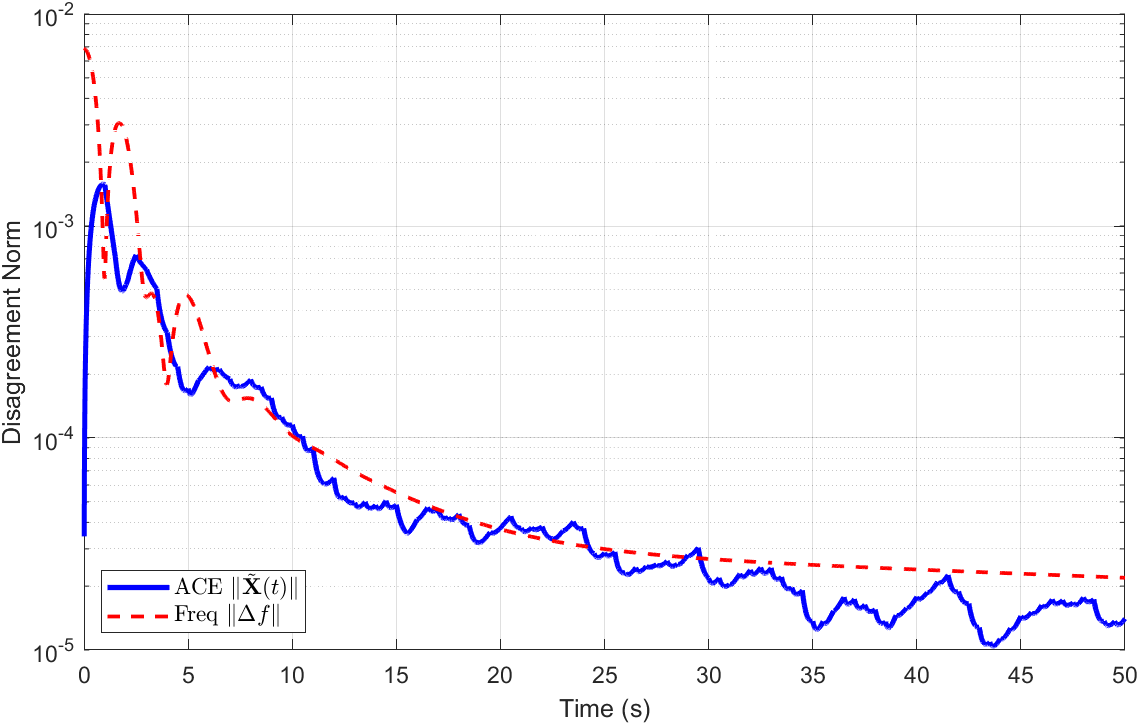}\\      
       {\footnotesize (a) Frequency and ACE deviations (consensus disagreement)}
        \label{fig:frequency_deviation}
    \end{minipage}
    \hfill
    \begin{minipage}[t]{0.45\textwidth}
        \centering
        \includegraphics[width=0.85\linewidth]{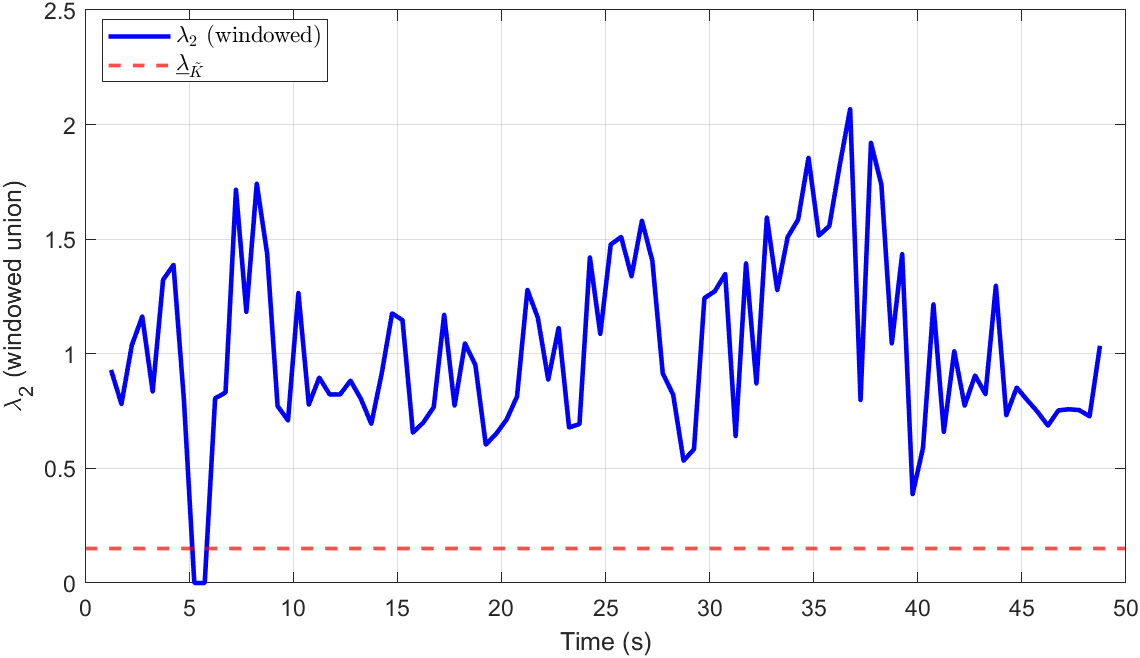}\\      %width=20 em,height=11em 
      {\footnotesize (b) Windowed graph algebraic connectivity}
        \label{fig:Windowed_Graph_Connectivity}
    \end{minipage}
    \caption{Control performance of the multi-agent 20-area LFC system under the proposed consensus-aware co-design.}
    \label{fig:power-metrics}
\end{figure*}

\subsection{Wireless Communication System}

We consider $N_c=20$ CAs and $N_m=40$ CUs.
Each CU communicates with the BS over orthogonal channels, while the CAs exchange their information directly via D2D links to achieve consensus. All CAs are statically located within an $80\,\mathrm{m}\times 80\,\mathrm{m}$ microgrid area, representing controllers of DER units such as rooftop solar inverters, battery storage systems, and controllable loads, whereas CUs are placed uniformly at random in the same area. All CAs transmit at the same power $P_{\mathrm{CA}}$, each CU transmits at uplink power $P_u^{\mathrm{UL}}$, and the BS transmits at downlink power $P_u^{\mathrm{DL}}$
per CU, as specified in Table~\ref{tab:numvalues}. We adopt a Rayleigh block-fading channel model, in which the channel coefficients remain constant within each coherence interval and vary independently across intervals. The considered slot duration $T_k=0.5$~s is determined by the
LFC secondary control timescale rather than the wireless coherence time. For a typical CU with coherence time $T_c\approx45$~ms (corresponding to pedestrian mobility at 2.4~GHz), each time slot contains approximately $T_k/T_c\approx11$ statistically independent fading realizations. Since multiple independent channel realizations occur within a single time slot, the achievable rate experienced during that slot is well approximated by its ergodic average over these realizations. Consequently, the per-slot transmission rate in our optimization is modeled using the ergodic capacity rather than an instantaneous channel realization. At the high SNRs typical of the
$80\,\mathrm{m}\times 80\,\mathrm{m}$ deployment, ergodic capacity closely approximates the capacity at the mean SNR, validating the per-slot Rayleigh model for CU links~\cite{goldsmith2005wireless}. The CA channels, by contrast, are stationary (DER controllers with fixed location) with coherence time far exceeding $T_s$, so the single per-slot realization is exact. The remaining simulation parameters are listed in
Table~\ref{tab:numvalues}.
%The control packet size $L_{\mathrm{pkt}}=400$~bits (50~bytes) is selected consistent with the IEC~61850 Type-1 minimum control message size~\cite{IEC61850}, accommodating the state payload, node identifier, slot index, and error-detection checksum.

\subsection{Numerical Results}
Here, numerical results are presented for the proposed co-design approach in the hybrid/switched system described above. The optimization is implemented in MATLAB using CVX.  The total time of simulation is $T = 50~\mathrm{s}$, which is discretized into slots $T_k,~ k\in \mathbb{N}$ with a length of $0.5~\mathrm{s}$. The continuous-time dynamics over all $T_k$ slots are integrated with a step size of $0.01~\mathrm{s}$. 
\begin{figure*}
\begin{minipage}{0.48\textwidth}
\begin{subfigure}{\textwidth}
\includegraphics[width=20 em,height=11em]{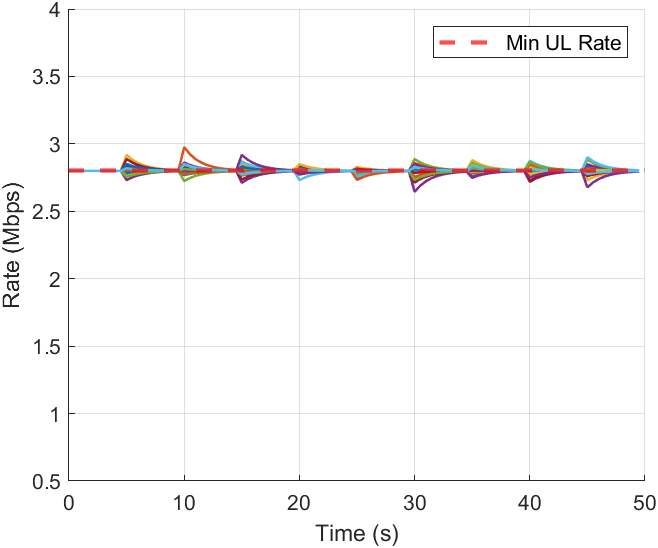}
\caption{\footnotesize Cellular uplink rates (First 20 CUs)}
\label{fig:UL-first-20}
\end{subfigure}
 \medskip
\begin{subfigure}{\textwidth}
\includegraphics[width=20 em,height=11em]{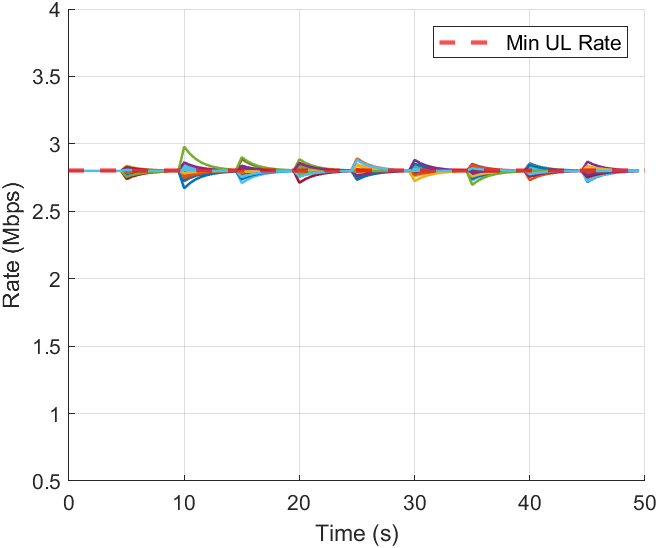}
\caption{\footnotesize Cellular uplink rates (Second 20 CUs)}
\label{fig:UL-second-20}
\end{subfigure}
\caption{Cellular uplink rate}
    \label{fig:Cellular-UL-Rates}
\end{minipage}
\hfill
\begin{minipage}{0.48\textwidth} 
\begin{subfigure}{\textwidth}
\includegraphics[width=20 em,height=11em]{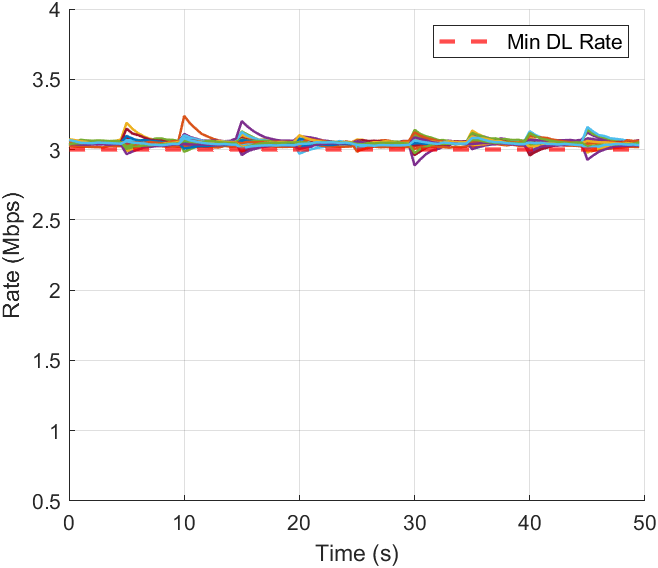}
\caption{\footnotesize  Cellular downlink rates (First 20 CUs)}
\label{fig:DL-first-20}
\end{subfigure}
 \medskip
\begin{subfigure}{\textwidth}
\includegraphics[width=20 em,height=11em]{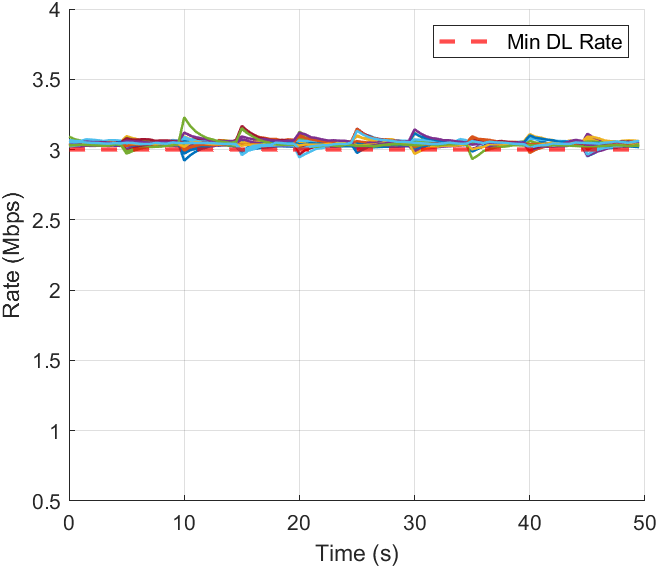}
\caption{\footnotesize Cellular downlink rates  (Second 20 CUs)}
\label{fig:DL-second-20}
\end{subfigure}
\caption{Cellular downlink rate}
    \label{fig:Cellular-DL-Rates}
    \end{minipage}
\end{figure*}

To evaluate the proposed co-design approach, we set $\tilde{K}=5$ and $\underline{\lambda}_{\tilde{K}}=0.15$. Fig.~\ref{fig:power-metrics}a shows the frequency deviations $\Delta f_i(t)$ of the 20~control areas and the ACE consensus disagreement
$\|\tilde{\mathbf{X}}(t)\|$ (defined in~\eqref{disagree}), which measures how far the areas are from reaching agreement on a common ACE correction signal. Following the step load disturbances, all areas experience frequency dips on the order of $10^{-2}$~Hz that remain bounded and converge to a common steady-state offset, consistent with the tie-line synchronization mechanism described in Section~\ref{sec:simulation_system}.
The ACE consensus disagreement $\|\tilde{\mathbf{X}}(t)\|$ decays and falls below $10^{-4}$ by $t\approx 10$~s, demonstrating that the joint connectivity and coupling conditions enforced through constraint~C$_1$ are sufficient to guarantee consensus in the presence of the time-varying wireless-induced switching topology and bandwidth constraints, consistent with Theorem~\ref{thm:consensus-windowed}.
In addition, Fig.~\ref{fig:power-metrics}b reports the windowed algebraic connectivity 
$\lambda_2\!\left(\sum_{\ell=k}^{k+\tilde K-1} \mathbf{L}_c(T_\ell)\right)$ 
of the CA communication graph. Since only a small number of D2D links are active in each slot, the instantaneous connectivity $\lambda_2(\mathbf{L}_c(T_k))$ may temporarily drop to zero when the graph becomes disconnected. In contrast, the windowed metric, evaluated over $\tilde K$ consecutive slots, as required by Theorem~1, remains strictly above the design threshold $\underline{\lambda}_{\tilde K}=0.15$ throughout the simulation. This shows that, although single-slot topologies may be disconnected, the union graph over each $\tilde K$-slot window is connected, ensuring satisfaction of the joint-connectivity constraint and guaranteeing consensus under switching topologies.

Moreover, Figs.~\ref{fig:Cellular-UL-Rates}–\ref{fig:Average_D2D_rate} illustrate the wireless performance when $\tilde{K}=5$ and $\underline{\lambda}_{\tilde{K}}=0.15$. Specifically, Figs.~\ref{fig:UL-first-20} and \ref{fig:UL-second-20} present the UL rates for the first and second groups of 20 CUs, respectively, along with their corresponding QoS thresholds. Similarly, Figs.~\ref{fig:DL-first-20} and \ref{fig:DL-second-20} show the DL rates for these two groups, also compared against their QoS thresholds. As observed, both UL and DL rates, $R_u^{\mathrm{UL,min}}$ and $R_u^{\mathrm{DL,min}}$, remain close to or above these targets for the entire time horizon despite random fading. This behavior demonstrates that, by using the proposed approach, the optimizer effectively balances CA consensus requirements with CU service guarantees under shared bandwidth constraints. Furthermore, Fig.~\ref{fig:Average_D2D_rate} illustrates the average CA--CA D2D
rate across active links. The mean value stays slightly above the target
$R_{\min}^{\mathrm{CA}} = 10~\mathrm{kbps}$ throughout the simulation, confirming that
the D2D links used in the consensus protocol always satisfy the 
QoS requirement with a small safety margin. As illustrated in Fig.~\ref{fig:Active_CA-CA_D2D}, in this configuration
(with $\tilde K = 5$ and $\underline{\lambda}_{\tilde K}=0.15$), the algorithm, on average, selects
$11.5$ CA--CA links, which is the minimum number of edges required for
feasibility under the chosen windowed spectral-connectivity constraint.
This demonstrates that the proposed graph–design stage indeed finds the sparsest admissible adjacency patterns that satisfy the windowed spectral condition while minimizing communication overhead. Overall, these results show that the proposed consensus-aware co-design yields (i) fast and accurate frequency regulation with guaranteed consensus among the CAs, (ii) satisfaction of CU QoS targets in both uplink and downlink, and (iii) significant bandwidth savings through minimal, spectrally constrained D2D topologies and adaptive resource allocation.
\begin{figure}[t]
  \centering
  \includegraphics[width=0.4\textwidth]{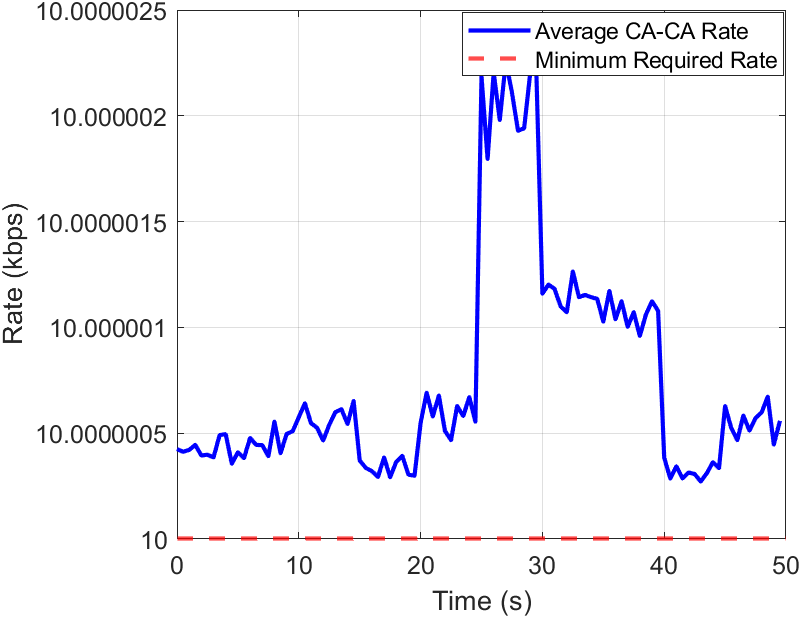}
  \caption{Average CA-CA D2D rate}
  \label{fig:Average_D2D_rate}
\end{figure}
\begin{figure}[t]
  \centering
  \includegraphics[width=0.9\linewidth]{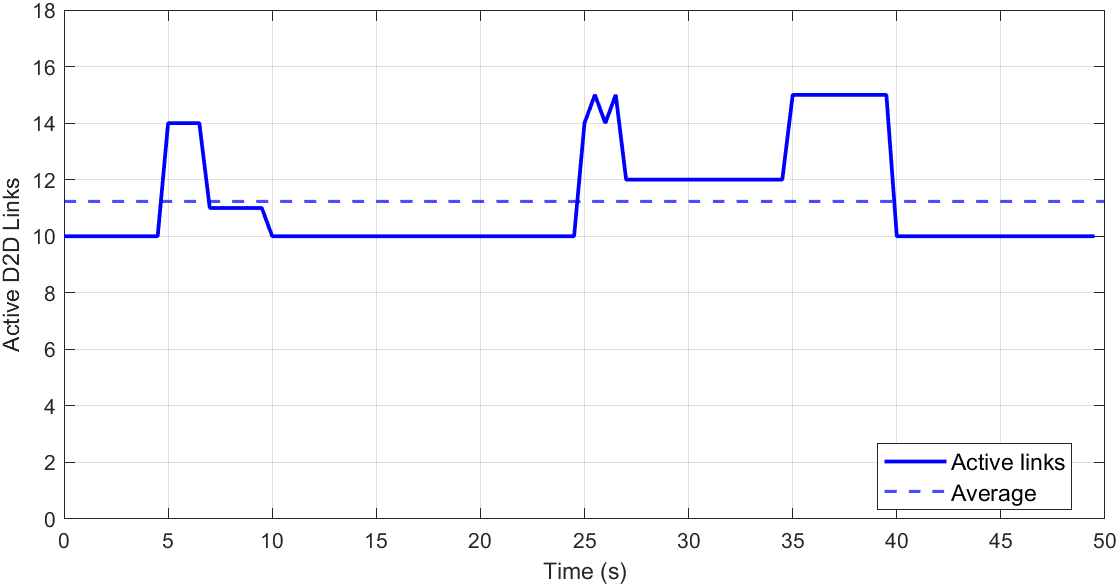}
  \caption{Active CA-CA D2D Links per Slot.}
  \label{fig:Active_CA-CA_D2D}
\end{figure}

\textbf{Comparison with two baseline methods}:
We compare the proposed co-design approach with the following two baselines:

1. {\em Event-Triggered  baseline~\cite{cheng2018event}}.
The event-triggered baseline implements the event-triggered consensus
protocol with a switching topology that alternates between
two complementary subgraphs $\mathcal{G}_A$ and
$\mathcal{G}_B$ every $\tau_d=3$ slots, satisfying the
jointly-connected condition over $\tilde{K}=5$ consecutive
slots.
Each agent broadcasts its information, including ACE integral state $\xi_i$, only
when the staleness error exceeds a predefined
threshold determined by the triggering function of
Cheng et al.~\cite{cheng2018event}; between trigger events, the last broadcast value is held constant (zero-order hold). Bandwidth is pre-provisioned conservatively without
per-link CSI or joint optimization with the cellular tier. Specifically, for D2D links, bandwidth is reserved for all
scheduled edges in $\mathcal{G}(T_k)$~--~not only those
whose endpoints trigger~--~using the worst-case channel SNR among all scheduled links in the current slot:
\begin{equation}
B_{ij}^{\mathrm{ET}}(T_k)
= \frac{2\,R_{\min}^{\mathrm{CA}}}
       {\omega\,
        \log_2\!\bigl(1 +
          \min_{(i,j)\in\mathcal{G}(T_k)} c_{ij}(T_k)
        \bigr)}.
\label{eq:ET-D2D-BW}
\end{equation}
Indeed, the fundamental limitation of
operating without a joint optimizer is the operator must either over-provision~--~wasting spectrum~--~or risk QoS violations.
%Bandwidth is pre-provisioned conservatively for all scheduled edges using the worst-case channel SNR among active links per slot, without per-link CSI or joint optimization with the cellular tier.
 
2. {\em Fully-Connected baseline}.
The fully-connected baseline activates all
$\binom{N_c}{2}=190$ D2D edges in every slot, providing the maximum possible algebraic connectivity
$\lambda_2(\mathbf{L})=N_c=20$ and representing the theoretical upper bound on consensus performance. Bandwidth is allocated conservatively using the worst-case channel SNR across all users per slot (identical to the event-triggered baseline) but applied to all 190 D2D links simultaneously. No topology optimization is performed.
 
Fig.~\ref{fig:ACE_Comp} illustrates the ACE consensus disagreement $\|\tilde{\mathbf{X}}(t)\|$ for the three schemes.
The co-design scheme achieves convergence to $1.4\times10^{-5}$ within approximately 15~s, demonstrating successful secondary frequency regulation across all 20~areas.
The fully-connected baseline converges faster and to a lower residual ($5.5\times10^{-7}$) owing to its maximum algebraic connectivity $\lambda_2=N_c=20$, confirming the
theoretical relationship between $\lambda_2$ and consensus rate discussed in
Section~\ref{sec:system-model}. The event-triggered baseline fails to achieve consensus, stabilizing at $3.9\times10^{-4}$~---~a level 28$\times$ above co-design~--~because its fixed switching topology provides weaker algebraic connectivity for 20~heterogeneous integrating ACE states. In addition, Fig.~\ref{fig:FREQ_Comp} shows the frequency consensus disagreement $\|\Delta\tilde{f}(t)\|$ for the three schemes.
All three converge to virtually identical final values (co-design: $2.2\times10^{-5}$, event-triggered: $2.3\times10^{-5}$, fully-connected: $2.2\times10^{-5}$), confirming that frequency synchronization is governed by the physical tie-line
coupling, which is consistent with the system model in Section~\ref{sec:simulation_system}.
 
Fig.~\ref{fig:Bandwith_comp} shows the corresponding normalized bandwidth consumption. The co-design scheme maintains a stable allocation of 59.1\%, substantially below both event-triggered~(83.2\%) and fully-connected~(86.1\%), achieving savings of 24.1\% and 27.0\%, respectively. The event-triggered scheme exhibits severe slot-to-slot fluctuations with spikes occasionally exceeding 100\% of the available
budget, a consequence of the conservative worst-case pre-provisioning applied to a time-varying set of scheduled links without per-slot channel-aware optimization.
The fully-connected scheme maintains a flat allocation at 86.1\%, reflecting the fixed cost of provisioning all 190~D2D
links simultaneously. Taken together, Figs.~\ref{fig:ACE_Comp} and~\ref{fig:Bandwith_comp} confirm the central claim of
this work: the proposed co-design achieves fully-connected quality consensus convergence at a bandwidth cost 27.0\% below fully-connected and 24.1\% below event-triggered, demonstrating the efficiency
of the joint topology and bandwidth optimization.

Notably, the event-triggered protocol was originally proposed to reduce communication overhead in control systems by transmitting only when the local estimation error exceeds a threshold. However, the simulation results reveal a fundamental limitation: without joint optimization of the control and communication, the event-triggered scheme
cannot achieve its goal. Specifically, its switching topology provides weaker algebraic connectivity
($\lambda_2\approx 0.094 < \underline{\lambda}_{\tilde{K}}=0.15$)
for 20~heterogeneous areas, causing the ACE
disagreement to stagnate at $3.9\times10^{-4}$ while consuming 83.2\% of the available bandwidth~--~24.1\% more than the proposed co-design. This demonstrates that event-triggered communication
alone is insufficient when the underlying wireless network is not co-designed with the control objective: reducing the number of transmissions.
\begin{figure}[h]
  \centering
  \includegraphics[width=0.8\linewidth]{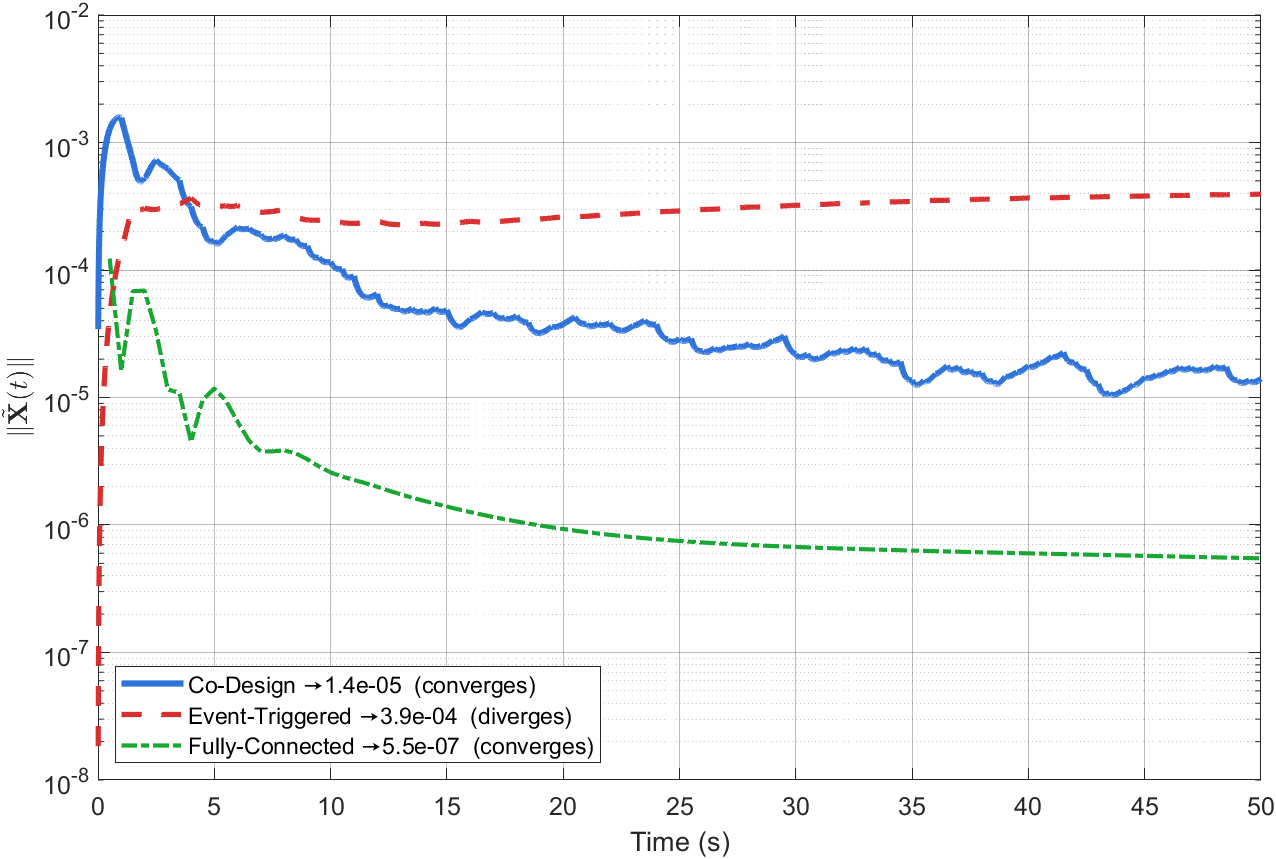}
  \caption{ACE deviations (Consensus disagreement) of the proposed co-design compared with the event-triggered and fully-connected baselines.}
  \label{fig:ACE_Comp}
\end{figure}
\begin{figure}[h]
  \centering
  \includegraphics[width=0.8\linewidth]{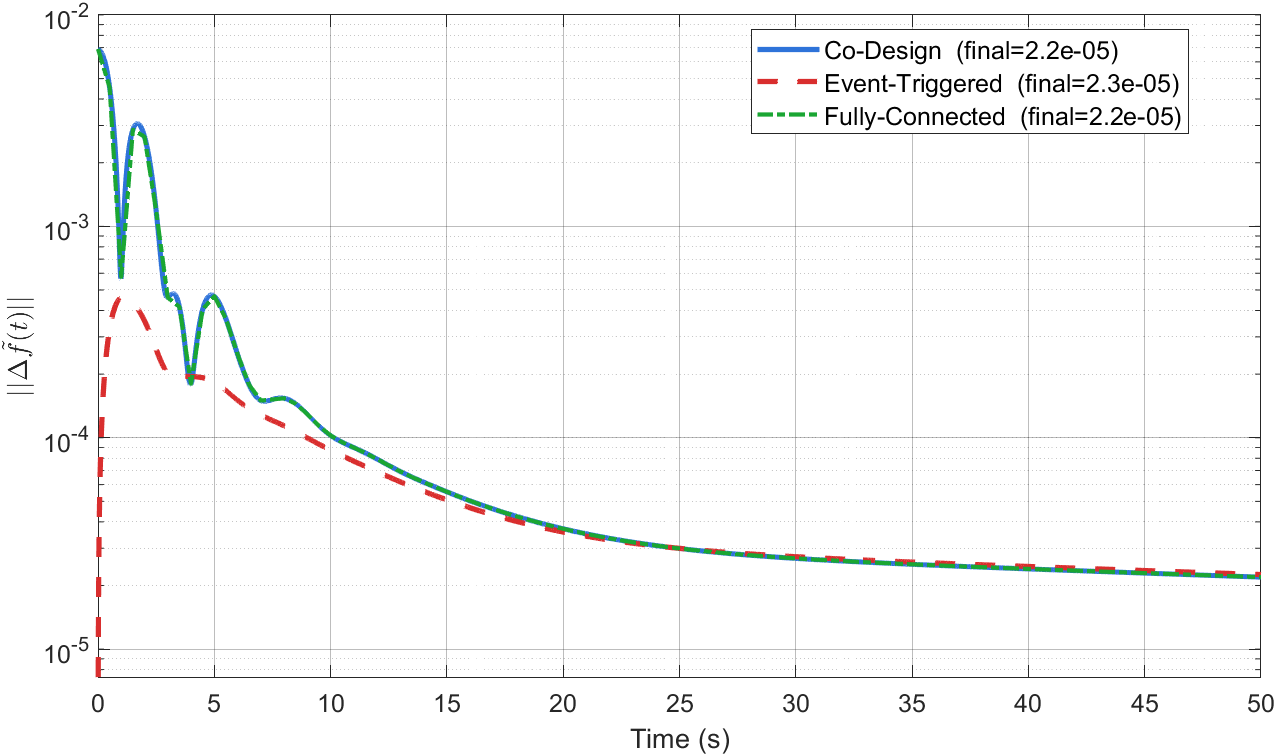}
  \caption{Frequency deviations of the proposed co-design compared with the event-triggered and fully-connected baselines.}
  \label{fig:FREQ_Comp}
\end{figure}

\begin{figure}[t]
  \centering
  \includegraphics[width=0.8\linewidth]{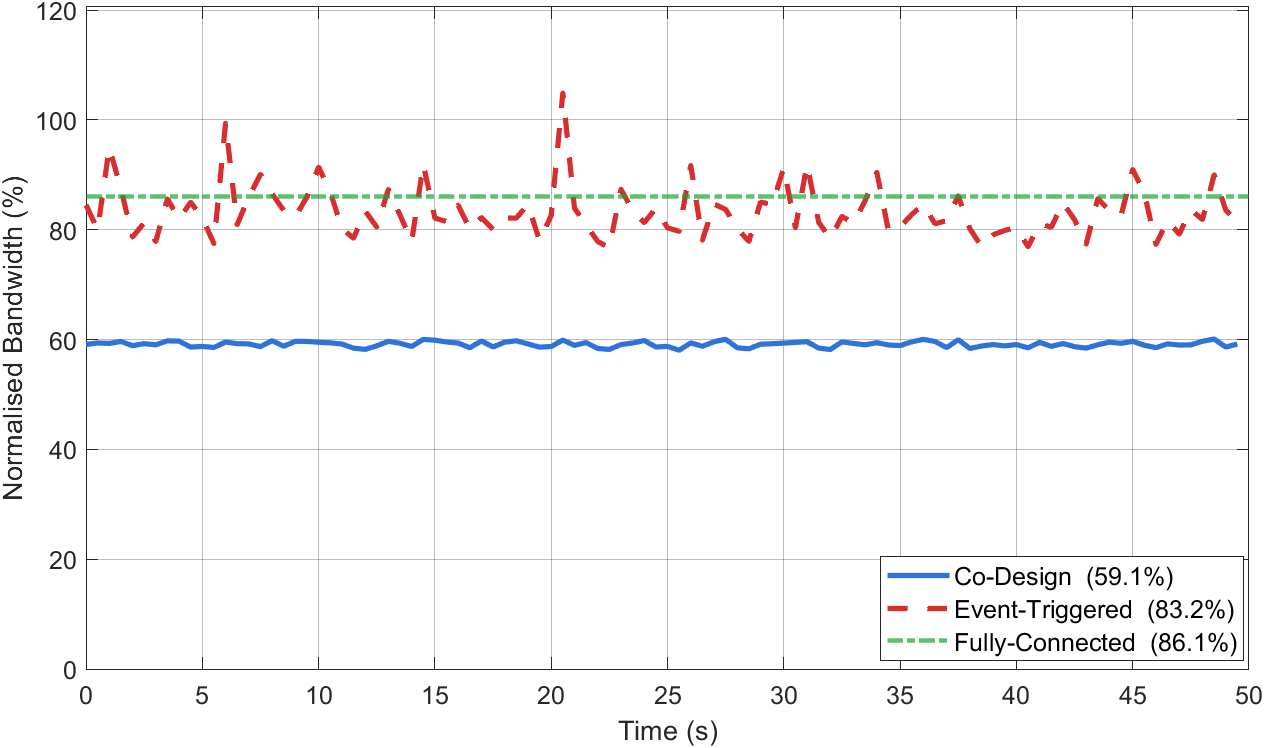}
  \caption{Normalized total bandwidth consumption of the proposed co-design compared with the event-triggered and fully-connected baselines.}
  \label{fig:Bandwith_comp}
\end{figure}
\vspace{-1.0em}
\section{Conclusion}\label{Conclusion}
We have introduced a consensus-aware co-design approach that jointly optimizes the MACS communication topology and wireless resource allocation. The proposed approach is motivated by the coupling between the communication and control subsystems in emerging MACSs, where inter-agent communication relies on D2D links sharing the wireless spectrum with CUs. By jointly optimizing the consensus topology and bandwidth allocation, the proposed approach minimizes total bandwidth consumption while maintaining the communication and consensus requirements of the MACS. In contrast, optimizing the communication and control subsystems independently either degrades consensus performance or wastes wireless resources. Although the resulting optimization problem is a MINP (NP-hard in general), a structural property of the optimal solution, formalized in \textbf{Lemma~\ref{lem:adjacency_minimal}}, shows that the feasible search space can be reduced without excluding any globally optimal solution. This structural property provides theoretical insight into the interaction between the communication and control subsystems by showing that their joint optimization can be solved in two sequential stages with polynomial complexity per slot. We validate the proposed approach through a smart power grid case study comprising 20 CAs (LFC areas) and 40 CUs. The proposed approach achieves 24–27\% bandwidth savings compared with the two baselines while maintaining consensus and satisfying all QoS requirements. %While learning-based methods are promising for solving high-complexity optimization problems, they are not pursued in this work for three reasons. First, the problem involves hybrid decision variables (continuous and discrete) together with nonconvex constraints, making training computationally expensive. Second, providing theoretical generalization guarantees under uncertainty remains challenging. Third, MACSs are often deployed in safety-critical applications where reliability and interpretability are essential, yet difficult to ensure for many learning-based approaches. Developing learning-based co-design frameworks that address these challenges constitutes an important direction for future research.
\bibliographystyle{ieeetr}  
\bibliography{references}

@article{park2017wireless,
  title={Wireless network design for control systems: A survey},
  author={Park, Pangun and Ergen, Sinem Coleri and Fischione, Carlo and Lu, Chenyang and Johansson, Karl Henrik},
  journal={IEEE Communications Surveys \& Tutorials},
  volume={20},
  number={2},
  pages={978--1013},
  year={2017},
  publisher={IEEE}
}

@article{chen2019control,
  title={On the control of multi-agent systems: A survey},
  author={Chen, Fei and Ren, Wei and others},
  journal={Foundations and Trends{\textregistered} in Systems and Control},
  volume={6},
  number={4},
  pages={339--499},
  year={2019},
  publisher={Now Publishers, Inc.}
}

@article{tiberi2013energy,
  title={Energy-efficient sampling of networked control systems over {IEEE} 802.15. 4 wireless networks},
  author={Tiberi, Ubaldo and Fischione, Carlo and Johansson, Karl Henrik and Di Benedetto, Maria Domenica},
  journal={Automatica},
  volume={49},
  number={3},
  pages={712--724},
  year={2013},
  publisher={Elsevier}
}

@article{nowzari2019event,
  title={Event-triggered communication and control of networked systems for multi-agent consensus},
  author={Nowzari, Cameron and Garcia, Eloy and Cort{\'e}s, Jorge},
  journal={Automatica},
  volume={105},
  pages={1--27},
  year={2019},
  publisher={Elsevier}
}

@article{pang2025communication,
  title={Communication-control codesign for large-scale wireless networked control systems},
  author={Pang, Gaoyang and Liu, Wanchun and Niyato, Dusit and Vucetic, Branka and Li, Yonghui},
  journal={IEEE Journal on Selected Areas in Communications},
   year={2025},
  volume={43},
  number={10},
  pages={3295-3312},
  publisher={IEEE}
}

@article{termehchi2020joint,
  title={Joint sampling time and resource allocation for power efficiency in industrial cyber--physical systems},
  author={Termehchi, Atefeh and Rasti, Mehdi},
  journal={IEEE Transactions on Industrial Informatics},
  volume={17},
  number={4},
  pages={2600--2610},
  year={2020},
  publisher={IEEE}
}

@article{zeng2019joint,
  title={Joint communication and control for wireless autonomous vehicular platoon systems},
  author={Zeng, Tengchan and Semiari, Omid and Saad, Walid and Bennis, Mehdi},
  journal={IEEE Transactions on Communications},
  volume={67},
  number={11},
  pages={7907--7922},
  year={2019},
  publisher={IEEE}
}

@article{diao2025task,
  title={Task-Oriented Co-Design of Communication, Computing, and Control for Edge-Enabled Industrial Cyber-Physical Systems},
  author={Diao, Yufeng and Zhang, Yichi and De Martini, Daniele and Zhao, Philip Guodong and Li, Emma Liying},
  journal={IEEE Journal on Selected Areas in Communications},
  year={2025},
  volume={43},
  number={9},
  pages={3041-3055},
  publisher={IEEE}
}

@article{ren2005consensus,
  title={Consensus seeking in multiagent systems under dynamically changing interaction topologies},
  author={Ren, Wei and Beard, Randal W},
  journal={IEEE Transactions on Automatic Control},
  volume={50},
  number={5},
  pages={655--661},
  year={2005},
  publisher={IEEE}
}

@article{moreau2005stability,
  title={Stability of multiagent systems with time-dependent communication links},
  author={Moreau, Luc},
  journal={IEEE Transactions on Automatic Control},
  volume={50},
  number={2},
  pages={169--182},
  year={2005},
  publisher={IEEE}
}

@article{wang2025communication,
  title={Communication and Control Co-Design for Heterogeneous Industrial IoT: A Logic-Based Stochastic Switched System Approach},
  author={Wang, Shuling and Zhu, Shanying and Chen, Cailian and Shen, Fei and Zhang, Weihai and Guan, Xinping},
  journal={IEEE Journal on Selected Areas in Communications},
  year={2025},
  publisher={IEEE}
}

@article{li2025distributed,
  title={Distributed Optimization for Estimation and Communication Co-Design Under Bandwidth Constraints},
  author={Li, Peizhe and Chen, Cailian and Zhu, Shanying and Guan, Xinping},
  journal={IEEE Transactions on Signal and Information Processing over Networks},
  year={2025},
  volume={11},
  number={},
  pages={1112-1126},
  publisher={IEEE}
}

@article{eisen2019control,
  title={Control aware radio resource allocation in low latency wireless control systems},
  author={Eisen, Mark and Rashid, Mohammad M and Gatsis, Konstantinos and Cavalcanti, Dave and Himayat, Nageen and Ribeiro, Alejandro},
  journal={IEEE Internet of Things Journal},
  volume={6},
  number={5},
  pages={7878--7890},
  year={2019},
  publisher={IEEE}
}

@ARTICLE{li2016event,
  author={Li, Huaqing and Chen, Guo and Huang, Tingwen and Dong, Zhaoyang and Zhu, Wei and Gao, Lan},
  journal={IEEE {Transactions} on {Cybernetics}}, 
  title={Event-Triggered Distributed Average Consensus Over Directed Digital Networks With Limited Communication Bandwidth}, 
  year={2016},
  volume={46},
  number={12},
  pages={3098-3110},
  publisher={IEEE}
  }

@inproceedings{hong2019adaptive,
  title={Adaptive communication and control co-design for multi-agent coordination with second-order dynamics},
  author={Hong, Jooeun and Votion, Johnathan and Cao, Yongcan and Jin, Yufang},
  booktitle={2019 American Control Conference (ACC)},
  pages={5322--5327},
  year={2019},
  organization={IEEE}
}

@article{oliveira2021co,
  title={Co-design of consensus-based approach and reliable communication protocol for vehicular platoon control},
  author={Oliveira, Ren{\^e} and Montez, Carlos and Boukerche, Azzedine and Wangham, Michelle S},
  journal={IEEE Transactions on Vehicular Technology},
  volume={70},
  number={9},
  pages={9510--9524},
  year={2021},
  publisher={IEEE}
}

@article{olfati2007consensus,
  title={Consensus and cooperation in networked multi-agent systems},
  author={Olfati-Saber, Reza and Fax, J Alex and Murray, Richard M},
  journal={Proceedings of the IEEE},
  volume={95},
  number={1},
  pages={215--233},
  year={2007},
  publisher={IEEE}
}

@book{ren2008distributed,
  title={Distributed consensus in multi-vehicle cooperative control: Theory and applications},
  author={Ren, Wei and Beard, Randal W},
  year={2008},
  publisher={Springer}
}

@article{jadbabaie2003coordination,
  title={Coordination of groups of mobile autonomous agents using nearest neighbor rules},
  author={Jadbabaie, Ali and Lin, Jie and Morse, A Stephen},
  journal={IEEE Transactions on Automatic Control},
  volume={48},
  number={6},
  pages={988--1001},
  year={2003},
  publisher={IEEE}
}

@inproceedings{scardovi2008synchronization,
  title={Synchronization in networks of identical linear systems},
  author={Scardovi, Luca and Sepulchre, Rodolphe},
  booktitle={2008 47th IEEE Conference on Decision and Control},
  pages={546--551},
  year={2008},
  organization={IEEE}
}

@inproceedings{ghosh2006growing,
  title={Growing well-connected graphs},
  author={Ghosh, Arpita and Boyd, Stephen},
  booktitle={Proceedings of the 45th IEEE Conference on Decision and Control},
  pages={6605--6611},
  year={2006},
  organization={IEEE}
}

@article{singh2017distributed,
  title={Distributed multi-agent system-based load frequency control for multi-area power system in smart grid},
  author={Singh, Vijay Pratap and Kishor, Nand and Samuel, Paulson},
  journal={IEEE Transactions on Industrial Electronics},
  volume={64},
  pages={5151--5160},
  year={2017},
  publisher={IEEE}
}

@article{olfati2004consensus,
  title={Consensus problems in networks of agents with switching topology and time-delays},
  author={Olfati-Saber, Reza and Murray, Richard M},
  journal={IEEE Transactions on Automatic Control},
  volume={49},
  number={9},
  pages={1520--1533},
  year={2004},
  publisher={IEEE}
}

@article{cheng2018event,
  title={Event-triggered consensus of homogeneous and heterogeneous multiagent systems with jointly connected switching topologies},
  author={Cheng, Bin and Wang, Xiangke and Li, Zhongkui},
  journal={IEEE transactions on cybernetics},
  volume={49},
  number={12},
  pages={4421--4430},
  year={2018},
  publisher={IEEE}
}

@book{goldsmith2005wireless,
  author    = {Goldsmith, A.},
  title     = {Wireless Communications},
  publisher = {Cambridge University Press},
  year      = {2005}
}

\end{document}